\documentclass[12pt,a4paper]{article}
\usepackage{graphicx}
\usepackage{amsmath,amsfonts,amsthm}
\usepackage{xfrac,esint,mathrsfs,color}
\usepackage[utf8]{inputenc}
\usepackage[colorlinks=true]{hyperref}

\newcommand{\Huno}{{\calH}^{1}}
\newcommand{\Rdue}{\R^2}
\newcommand{\uabn}{u_{x_0}}

\newcommand\res{\mathop{\hbox{\vrule height 7pt width .5pt depth 0pt
\vrule height .5pt width 6pt depth 0pt}}\nolimits}
\newcommand{\Ln}{{\calL}^2}

\newcommand\eps{\varepsilon}

\newcommand\diam{\mathrm{diam\,}}
\newcommand\Id{\mathrm{Id}}

\newcommand\conv{\mathrm{conv\,}}

\newcommand\dist{\mathrm{dist}}

\newcommand\R{\mathbb{R}}
\newcommand\N{\mathbb{N}}

\newcommand\calH{\mathcal{H}}
\newcommand\calF{\mathcal{F}}

\newcommand\calA{\mathcal{A}}

\newcommand\calB{\mathcal{B}}
\newcommand\calL{\mathcal{L}}
\newcommand{\LM}[1]{\hbox{\vrule width.2pt \vbox to#1pt{\vfill \hrule width#1pt height.2pt}}}
\newcommand{\LL}{{\mathchoice{\,\LM7\,}{\,\LM7\,}{\,\LM5\,}{\,\LM{3.35}\,}}}
\newcommand{\ba}[1]{\begin{eqnarray} #1 \end{eqnarray}}
\newcommand{\be}[1]{\begin{equation} #1 \end{equation}}
\newcommand{\bm}[1]{\begin{multline} #1 \end{multline}}
\newcommand{\bes}[1]{\begin{eqnarray*} #1 \end{eqnarray*}}
\newtheorem{theorem}{Theorem}[section]

\newtheorem{proposition}[theorem]{Proposition}
\newtheorem{lemma}[theorem]{Lemma}

\newtheorem{remark}[theorem]{Remark}

\definecolor{green}{rgb}{0,.5,0}

\renewcommand\theenumi{{(\roman{enumi})}}
\renewcommand\labelenumi{\theenumi}
\numberwithin{equation}{section}
\newcounter{Nummer}

\usepackage{fancyhdr}
\begin{document}
\begin{center}
  {\Large
Integral representation for functionals\\[1mm]
defined on $SBD^p$
in dimension two}\\[5mm]
{\today}\\[5mm]
Sergio Conti$^{1}$, Matteo Focardi$^{2}$, and Flaviana Iurlano$^{1}$\\[2mm]
{\em $^{1}$
 Institut f\"ur Angewandte Mathematik,
Universit\"at Bonn\\ 53115 Bonn, Germany}\\[1mm]
{\em $^{2}$ DiMaI, Universit\`a di Firenze\\ 50134 Firenze, Italy}\\[3mm]
    \begin{minipage}[c]{0.8\textwidth}
    We prove an integral representation result for functionals with growth
    conditions which give coercivity on the space $SBD^p(\Omega)$, for $\Omega\subset\R^2$.
    The space $SBD^p$ of functions
    whose distributional strain is the sum of an $L^p$ part and a bounded measure
    supported on a set of finite $\calH^{1}$-dimensional measure appears naturally in the
    study of fracture and damage models. Our result is based on the construction of
    a local approximation by $W^{1,p}$ functions. We also obtain a generalization
    of Korn's inequality in the $SBD^p$ setting.    
    \end{minipage}
\end{center}


\def\calM{\mathcal{M}}
\section{Introduction}

The direct methods of $\Gamma$-convergence are of paramount importance in studying  
variational limits and relaxation problems since their introduction in the seminal paper by 
Dal Maso and Modica \cite{DalMasoModica80}. They focus on the study of 
abstract limiting functionals $F(u,A)$, obtained for instance using
$\overline\Gamma$-convergence arguments; one key ingredient is the proof
of an integral representation for $F(u,A)$.
Here $u:\Omega\to\R^N$ is an element of a suitable function 
space $\mathscr{X}(\Omega)$, and $A$ runs in the class $\mathcal{A}(\Omega)$ of open subsets of a given open set 
$\Omega\subset\R^n$.  
The notion of \emph{variational functional} is at the heart of the matter: $F$, regarded as depending 
on the couple $(u,A)\in \mathscr{X}(\Omega)\times\mathcal{A}(\Omega)$, has to satisfy suitable lower semicontinuity, 
locality and measure theoretic properties (for more details see properties (i)-(iii) in Theorem~\ref{theorepr}).
The specific growth conditions of the functional determine 
the natural functional space in which the function $u$ lies.
Under these assumptions $F(u,A)$ can be written as an integral over the domain of integration $A$
with respect to a suitable measure.
The integrands may depend on $x$, $u(x)$ and $\nabla u(x)$, and possibly on other local quantities of $u$, 
such as higher order or distributional derivatives. 
Furthermore, as first shown in some cases in \cite{DalMasoModica86} and then generalized in \cite{bou-fon-masc},
the corresponding energy densities can be characterized in terms of cell formulas, i.e.~asymptotic Dirichlet 
problems on small cubes or balls involving $F$ itself, with boundary data depending on the local properties of $u$.

 
Integral representation results have been obtained in several contexts with increasing generality: 
starting with the pioneering contribution by De Giorgi for limits of area-type integrals \cite{DeGiorgi75}, 
it has been extended to functionals defined first on Sobolev spaces in \cite{Sbordone1975,CarboneSbordone1979,ButtazzoDalmaso80,ButtazzoDalmaso1985b,ButtazzoDalmaso1985} 
and on the space of functions with Bounded Variation 
in \cite{DalMaso80,BouchitteDalMaso93}, and then to energies defined on partitions in \cite{AmbrosioBraides1990a} 
and on the subspace $SBV$ in \cite{BraidesChiadopiat1996} 
(we refer to \cite{ButtazzoDalmaso1985,dalmaso,bou-fon-masc,bou-fon-leo-masc} for a more exhaustive list of references).
The global method for relaxation introduced and developed
in \cite{bou-fon-masc,bou-fon-leo-masc} provides a general approach that  unifies and extends 
the quoted results.

We address the integral representation of functionals defined 
on the subspace $SBD^p(\Omega)$ of the space $BD(\Omega)$ in two dimensions. 
The space of functions of bounded deformation  $BD(\Omega)$ is 
characterized by the fact that the symmetric part of the distributional gradient 
$Eu:=(Du+Du^T)/2$ of $u\in L^1(\Omega,\R^n)$ is a bounded Radon measure, namely
\begin{equation*}
 BD(\Omega):=\{u\in  L^1(\Omega;\R^n): Eu \in \calM(\Omega;\R^{n\times n}_\mathrm{sym})\},
\end{equation*}
where $\Omega\subseteq\R^n$ is an open set, see \cite{Suquet1978a,Temam1983,AmbrosioCosciaDalmaso1997}.
$BD$ and its subspaces $SBD$ and $SBD^p$ constitute the natural setting
for the study of plasticity,
damage and fracture models in a geometrically linear framework
\cite{Suquet1978a,Temam1983,TemamStrang1980,AnzellottiGiaquinta1980,KohnTemam1983}.
In particular, $SBD^p$ is the set of $BD$ functions such that the strain $Eu$ can be written
as the sum of a function in  $L^p(\Omega,\R^{n\times n})$ and a part concentrated on a rectifiable set with 
finite $\calH^{n-1}$-measure, see 
\cite{BellettiniCosciaDalmaso1998,Chambolle2004a,Chambolle2004b,cha-gia-pon}.

For functionals with linear growth defined on $SBD$ an integral representation result was obtained by
Ebobisse and Toader \cite{EboToa03}. These functionals, however, lack coercivity on the relevant space.
The situation of functionals defined on $SBD^p$ and with corresponding growth properties is open. 
We give here a solution in two dimensions.  
\begin{theorem}\label{theorepr}
 Let $\Omega\subset\R^2$ be a bounded Lipschitz set, $p\in (1,\infty)$,
 $F:SBD^p(\Omega)\times \calB(\Omega)\to[0,\infty)$ be such that
 \begin{enumerate}
   \item\label{theoreprborel} $F(u,\cdot)$ is a Borel measure for any $u\in SBD^p(\Omega)$;
  \item\label{theoreprlsc} $F(\cdot, A)$ is lower semicontinuous with respect to the strong $L^1(\Omega,\R^2)$-convergence for any open set $A\subset\Omega$;
  \item\label{theoreprlocal} $F(\cdot, A)$ is local for any open set $A\subset\Omega$;
  \item\label{theoreprgrowth} There are $\alpha,\beta>0$ such that for any $u\in SBD^p(\Omega)$, any $B\in \calB(\Omega)$,
\begin{align}
   &\alpha (\int_B |e(u)|^pdx + \int_{J_u\cap B} (1+|[u]|) d\Huno) \le
    F(u,B)\notag\\
    \le &
  \beta (\int_B (|e(u)|^p+1)dx + \int_{J_u\cap B} (1+|[u]|) d\Huno).\label{e:growthF}
   \end{align}  
 \end{enumerate}
Then there are two Borel functions $f:\Omega\times \R^2\times \R^{2\times 2}\to [0,\infty)$ and
$g:\Omega\times\R^2\times \R^2\times S^1\to [0,\infty)$ such that
\begin{equation}\label{e:fg}
 F(u,B)=\int_B f(x,u(x),\nabla u(x)) dx + \int_{B\cap J_u} g(x,u^-(x), u^+(x), \nu_u(x)) d\Huno\,.
\end{equation}
\end{theorem}
Above and throughout the paper we will refer to the book \cite{ambrosio} and to the papers 
\cite{AmbrosioCosciaDalmaso1997, BellettiniCosciaDalmaso1998} for the notation and results about 
$BV$ and $BD$ spaces, respectively. In particular, $\calB(\Omega)$ is the family of Borel subsets of $\Omega$.

The proof of Theorem~\ref{theorepr}, which is given in Section \ref{s:intrep1}, 
follows the general strategy introduced in \cite{bou-fon-masc,bou-fon-leo-masc}.
Their approach was based on a Poincar\'e-type inequality in $SBV$ by 
De Giorgi, Carriero and Leaci, which is not known in $SBD^p$ (see \cite{DeGiorgiCarrieroLeaci,ambrosio}).
Our main new ingredient
is the construction of an approximation by $W^{1,p}$ functions,
discussed in Section \ref{sec:approximation}, which permits to bypass the De Giorgi-Carriero-Leaci inequality.
The approximation is done so that the function is only modified outside a countable set of balls 
with small area and perimeter.
In each ball, we give  a construction of a $W^{1,p}$ extension for the $SBD^p$ function
by constructing a finite-element approximation on a countable mesh, which is chosen 
depending on the function $u$, see Section \ref{sec:grid}.

Our $W^{1,p}$ approximation result also leads naturally to the proof of the following variant of Korn's inequality for
$SBD^p$ functions.
\begin{theorem}\label{theo:korn}
Let $\Omega\subset\R^2$ be a connected, bounded, Lipschitz set and let $p\in (1,\infty)$. Then there exists a constant
$c$, depending on $p$ and $\Omega$, with the following property: for every $u\in SBD^p(\Omega)$ there exist a set $\omega\subset\Omega$
of finite perimeter, with $\Huno(\partial \omega)\leq c\Huno(J_u)$, and an affine function $a(x)=Ax+b$, 
with $A\in\R^{2{\times}2}$ skew-symmetric and $b\in\R^2$, such that
\bes{
\|u-a\|_{L^p(\Omega\setminus \omega,\R^2)}\leq c \|e(u)\|_{L^p(\Omega,\R^{2{\times}2})},\\
\|\nabla u-A\|_{L^p(\Omega\setminus \omega,\R^{2{\times}2})}\leq c \|e(u)\|_{L^p(\Omega,\R^{2{\times}2})}.
}
\end{theorem}
This improves a result of \cite{friedrich} to the sharp exponent. Variants of the first inequality 
were first obtained in \cite{ChambolleContiFrancfort,friedrich1}.

The construction of Section \ref{sec:approximation} turns out to be crucial also in proving existence for the Griffith fracture model,
generalizing \cite{DeGiorgiCarrieroLeaci} to the case of linearized elasticity in dimension two.
This will be the object of the forthcoming paper \cite{ContiFocardiIurlano_exist}.

\section{Approximation of \texorpdfstring{$SBD^p$}{SBDp} functions with small jump set}
\label{sec:grid}
\newcommand\n{\nu}
\newcommand\Hu{\mathcal{H}^1}
\newcommand\Lu{\mathcal{L}^1}
\newcommand\de{\delta}
\newcommand\Ri{R_i}
\newcommand\G{\mathcal{G}}
\newcommand\F{\mathcal{F}}
\newcommand{\average}{{\mathchoice {\kern1ex\vcenter{\hrule height.4pt width 6pt depth0pt} \kern-9.7pt} {\kern1ex\vcenter{\hrule height.4pt width 4.3pt depth0pt} \kern-7pt} {} {} }}
\def\Xint#1{\mathchoice
{\XXint\displaystyle\textstyle{#1}}%
{\XXint\textstyle\scriptstyle{#1}}%
{\XXint\scriptstyle\scriptscriptstyle{#1}}%
{\XXint\scriptscriptstyle\scriptscriptstyle{#1}}%
\!\int}
\def\XXint#1#2#3{{\setbox0=\hbox{$#1{#2#3}{\int}$}
\vcenter{\hbox{$#2#3$}}\kern-.5\wd0}}
\def\mint{\Xint-}

In this Section we prove the following approximation result.
\begin{theorem}\label{t:tecnico}
 Let $n=2$, $p\in[1,\infty)$. There exists $\eta>0$ and $\tilde{c}>0$ such that if $J\in\mathcal{B}(B_{2r})$, for some $r>0$, satisfies
 \be{\label{eq:Jsmall}\Hu(J)< 2r\eta,}
 then there exist $R\in (r,2r)$ for which the following holds: 
 for every $u\in SBD^p(B_{2r})$ with $\Hu(J_u\cap B_{2r}\setminus J)=0$ there exists $\phi(u)\in SBD^p(B_{2r})\cap W^{1,p}(B_R,\R^2)$ such that
\begin{enumerate}
  \item\label{tecnico1} $\Hu(J_u\cap \partial B_R)=0$;
  \item\label{tecnico2} $\displaystyle{\int_{B_R}|e(\phi(u))|^qdx\leq \tilde{c}\int_{B_R}|e(u)|^qdx}$, for every $q\in[1,p]$;
  \item\label{tecnico3} $\|u-\phi(u)\|_{L^1(B_R,\R^2)}\leq \tilde{c} R|Eu|(B_R)$;
  \item\label{tecnico4} $u=\phi(u)$ on $B_{2r}\setminus B_R$,  $\Hu(J_{\phi(u)}\cap \partial B_R)=0$;
  \item\label{tecnico5} if $u\in L^\infty(B_{2r},\R^2)$, then $\|\phi(u)\|_{L^\infty(B_{2r},\R^2)}\leq \sqrt{2}\, \|u\|_{L^\infty(B_{2r},\R^2)}$.
 \end{enumerate}

 \end{theorem}

 \begin{proof}
  
We follow an idea of \cite{ContiSchweizer1,ContiSchweizer2}. 

Arguing as in \cite[Lemma 4.3]{ContiSchweizer1} we first claim that there exists $R\in (r,2r)$ such that for 
$\de_k:= R\,2^{-k}$ we have
\ba{\label{eq:inter}&\Hu({J}\cap\partial B_R)=0,\\ 
\label{eq:corona}&\Hu({J}\cap(B_R\setminus B_{R-\de_k}))<20\eta\de_k, \quad \textrm{for every $k\in\N$}.}
To prove this, we first observe that (\ref{eq:inter}) holds for almost every $R$, therefore it suffices to show that 
(\ref{eq:corona}) holds on a set of positive measure. 
We consider the family of intervals
$$\{[R-\de_{k},R]: 
\Hu({J}\cap(B_R\setminus B_{R-\de_k}))\ge 20\eta\de_{k}\}$$
and we define ${I}$ as the union of all intervals of the family,
with $R\in (r,2r)$, $k\in\N$. By Vitali's covering theorem, there
exists a countable set $(R_i, k_i)_{i\in\N}$ such that the corresponding intervals $[\Ri-\de_{k_i},\Ri]
$ 
are pairwise disjoint and cover at least one fifth of ${I}$.
Therefore by \eqref{eq:Jsmall}  we obtain
$$2r\eta>\Hu({J}\cap B_{2r})\geq \sum_{i\in \N}\Hu({J}\cap (B_{R_i}\setminus B_{R_i-\de_{k_i}}))\geq \sum_{i\in \N} 20\eta \de_{k_i}.$$
Since $\de_{k_i}=\Lu([\Ri-\de_{k_i},\Ri])$,
we conclude that $\Lu({I})<r$. This proves the existence of $R$ such that
(\ref{eq:inter}) and (\ref{eq:corona}) hold, which is fixed for the rest of the proof.

\medskip
We define $R_k:= R-\de_k$ and $\overline{x}_{k,j}:=R_k(\cos\frac{2\pi j}{2^k},\sin\frac{2\pi j}{2^k})$, $j=1,\dots,2^k$. 
We say that $\overline{x}_{k,j}$ and $\overline{x}_{k',j'}$ are neighbors if either $k=k'$ and $j= j'\pm 1$, working modulo $2^k$, 
or {(up to a permutation)} $k=k'+1$ and $j\in \{2j'-1,2j',2j'+1\}$, again modulo $2^k$. This gives a decomposition of $B_R$ into countably many triangles, 
whose angles are uniformly bounded away from 0 and $\pi$, see Figure \ref{fig:grid1}.

\begin{figure}
 \begin{center}
  \includegraphics[height=5cm]{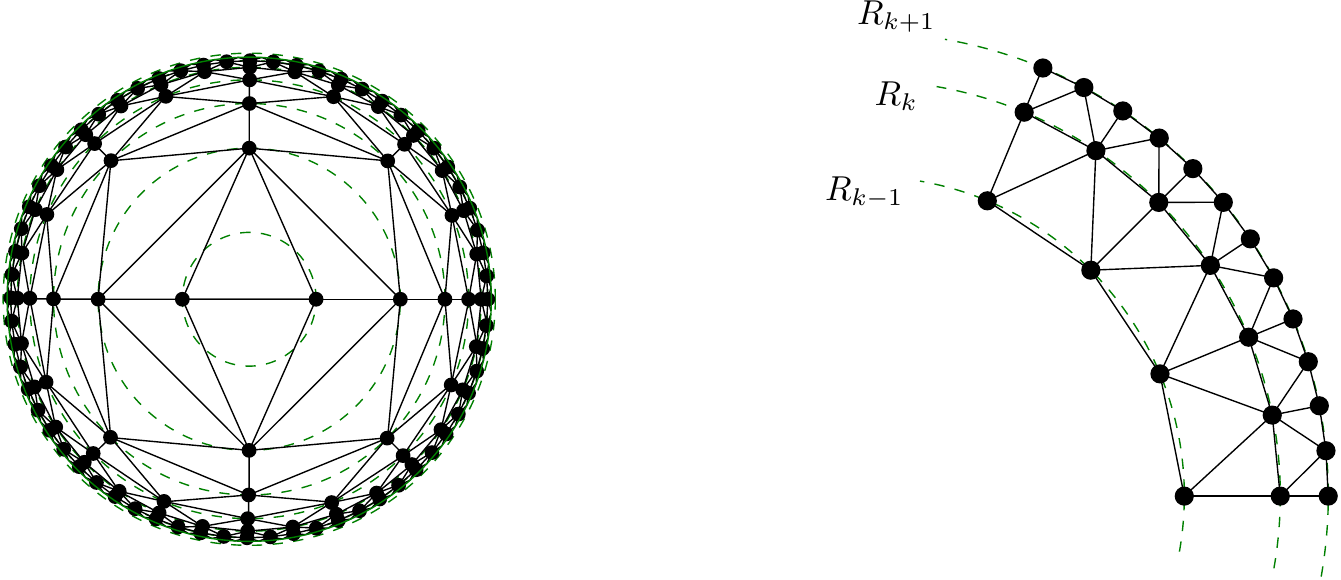}
\end{center}
\caption{Sketch of the construction of the grid in the proof of Theorem \ref{t:tecnico}.}
\label{fig:grid1}
\end{figure}

We shall construct $\phi(u)$ as a linear interpolation on a triangulation whose vertices are slight modifications of $\overline{x}_{k,j}$. 
Following the idea of \cite[Proposition 2.2]{ContiSchweizer2}, we next show how to construct the modified triangulation.
We start off considering two neighboring points $\overline{x}$ and $\overline{y}$ in $\{\overline{x}_{k,j}\}_{k,j}$, connected by the segment
$S_{\overline x,\overline y}\subset \overline B_{R_{k+1}}\setminus B_{R_{k-1}}$ for some $k$, and notice that $c_1\de_k\leq|\overline{x}-\overline{y}|\leq c_2\de_k$ 
for some $c_1\in(0,1)$, $c_2>1$ independent from $k$.
Let $\alpha:=c_1/(8c_2)$ and consider the convex envelope
\be{\label{eq:C}
O_{\overline{x},\overline{y}}:=\conv(B(\overline{x},\alpha\de_k)\cup B(\overline{y},\alpha\de_k)).}
Let $a_{\overline{x},\overline{y}}$ denote the infinitesimal rigid movement appearing in the Poincar\'{e}'s inequality for $u$ on 
the set 
\begin{equation*}
Q_{\overline{x},\overline{y}}:=\{\xi\in B_R:\dist(\xi,S_{\overline{x},\overline{y}})<|\overline{x}-\overline{y}|/{(8c_2)}\}.
\end{equation*}

Given $\vartheta\in(0,1)$, let us prove that for $\eta$ sufficiently small and $\tilde{c}$ sufficiently large, depending only on $\vartheta$, there exists 
a subset $F\subset B(\overline{x},\alpha\de_k){\times} B(\overline{y},\alpha\de_k)$
with $\frac{(\calL^2\times\calL^2)(F)}{\calL^2(B_{\alpha\de_k})^2}< \vartheta$, such that for every $(x,y)\notin 
F$ the one-dimensional section $u_z^\nu$ has the 
following properties:
\begin{enumerate}
\renewcommand\theenumi{{(P\arabic{enumi})}}
\renewcommand\labelenumi{\theenumi}
\item\label{tecp1} $u^{\n}_{z}\in SBV( s_{x,y} )$;
 \item\label{tecp2} $\calH^0(J_{u^{\nu}_{z}})=0$, so that $u^{\nu}_z\in W^{1,1}( s_{x,y} )$;
 \item\label{tecp3} $\displaystyle{\int_{ s_{x,y} }|(u^{\nu}_{z})'|dt\leq \frac{\tilde{c}}{\de_k}\int_{O_{\overline{x},\overline{y}}}|e(u)|dx'}$;
 \item\label{tecp4} $\displaystyle{|u(\xi)-a_{\overline{x},\overline{y}}(\xi)|\leq \frac{\tilde{c}}{\de_k}|Eu|(Q_{\overline{x},\overline{y}})}$, 
\textrm{ for $\xi=x,y$.}
\end{enumerate} 
Here $u^{\nu}_z(t):=u(z+t{\nu})\cdot {\nu}$ is the slice of $u$ along the line of direction 
\be{\label{eq:n}{\nu}:=\frac{x-y}{|x-y|},} 
and passing through
\be{\label{eq:z}z:=(\Id-\nu\otimes\nu)x \in \R\nu^{\perp}\cap (x+\R \nu)}
where $\R\nu^\perp$ is the linear space orthogonal to $\nu$,
and $s_{x,y}\subset\R$ is the segment such that $z+s_{x,y}\nu=S_{x,y}$,
see Figure \ref{fig:segment}.
\begin{figure}
 \begin{center}
  \includegraphics[height=5cm]{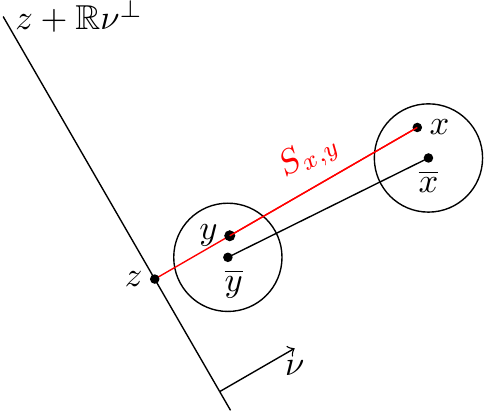}
\end{center}
\caption{{Slice along the line of direction $\nu=(x-y)/|x-y|$ passing through $z$ in the proof of Theorem~\ref{t:tecnico}.}}
\label{fig:segment}
\end{figure}

In order to obtain property \ref{tecp1} we first define the measure $\mu_{ \nu ,z}:=|u^ \nu _z|\calL^1$ and we observe that
\begin{align}\label{eq:sez}
\int_{B(\overline{x},\alpha\de_k){\times} B(\overline{y},\alpha\de_k)}&\mu_{ \nu ,z}( s_{x,y} )dx\,dy\\
&\leq
c\de_k^2\int_{B(\overline{x},\alpha\de_k)}dx\int_{\{| \nu -\overline{ \nu }|\leq c(\alpha)\}} \mu_{ \nu ,z}((O_{\overline{x},\overline{y}})^{\nu}_z)d \calH^1(\nu) ,  \nonumber
\end{align}
by the change of variables $y=x+s \nu $, where $\overline{\nu}$ is defined as $\nu$ with $\overline{x},\overline{y}$ in place of $x,y$ { and $(O_{\overline{x},\overline{y}})^\nu_z\subset\R$ corresponds to $O_{\overline{x},\overline{y}}\cap (z+\R\nu)$.} 
By Fubini's
theorem the last term in the previous inequality is less than or equal to
\be{\label{eq:sez2}
c\de_k^3\int_{\{| \nu -\overline{ \nu }|\leq c(\alpha)\}}d \calH^1(\nu) \int_{\nu^\perp}\mu_{ \nu ,z}(
(O_{\overline{x},\overline{y}})^ \nu _z)d\Hu(z)\leq c\de_k^3\mu(O_{\overline{x},\overline{y}}),
}
where $\mu(O_{\overline{x},\overline{y}}):=\int_{O_{\overline{x},\overline{y}}}|u|dx'$. By \eqref{eq:sez} and \eqref{eq:sez2} the set $F_1\subset B(\overline{x},\alpha\de_k){\times} B(\overline{y},\alpha\de_k)$ of points $(x,y)$, for which the inequality
$$\mu_{ \nu ,z}( s_{x,y} )>\frac{\tilde{c}}{\de_k}\mu(O_{\overline{x},\overline{y}})$$
holds, satisfies $\frac{(\calL^2\times\calL^2)(F_1)}{\calL^2(B_{\alpha\de_k})^2}< \vartheta/ {16}$,
for $\tilde{c}$ large enough, depending on $\vartheta$ and $\alpha$.

For $(x,y)\in B(\overline{x},\alpha\de_k){\times} B(\overline{y},\alpha\de_k)\setminus F_1$ we now repeat the argument 
in \eqref{eq:sez} and \eqref{eq:sez2} above redefining
$$\mu_{ \nu ,z}:=|D(u^ \nu _z)|.$$
We find that for $(x,y)$ out of a small (in the previous sense) set $F'_1$ one has $u^ \nu _z\in SBV( s_{x,y} )$ and
$$|D(u^ \nu _z)|( s_{x,y} )\leq \frac{\tilde{c}}{\de_k}|Eu|(O_{\overline{x},\overline{y}}).$$

As for property \ref{tecp2}, we use \eqref{eq:sez} and \eqref{eq:sez2} with $\mu_{ \nu ,z}:=\calH^0\res (J_{u^ \nu _z}\cap s_{x,y} )$ and 
$\mu:=\Hu\res (J_u\cap O_{\overline{x},\overline{y}})$.
Now \eqref{eq:Jsmall} implies 
$$\int_{B(\overline{x},\alpha\de_k){\times} B(\overline{y},\alpha\de_k)}\calH^0(J_{u^ \nu _z}\cap s_{x,y} )dx\,dy\leq c\de_k^4\eta,$$
and hence the set $F_2$ of points $(x,y)$ for which $\calH^0(J_{u^ \nu _z}\cap s_{x,y} )>1/2$ is also small in the previous sense, for $\eta$ small enough.
Note that this is the only step which requires the hypothesis on the dimension $n=2$.

Analogously properties \ref{tecp3} and \ref{tecp4} can be derived. From the argument above it is straightforward that for many points $x\in B(\overline{x},\alpha\de_k)$, 
still in the sense of a large $\vartheta$-fraction of $B(\overline{x},\alpha\de_k)$, there are many points $y\in B(\overline{y},\alpha\de_k)$ for which $(x,y)\notin F$. 

\medskip

Let us construct now the modified grid with an iterative process (see also \cite[Proposition 3.4]{ContiSchweizer2}). 
We will use the notation $B_i$ to indicate the balls $B(\overline{x}_{k,j},\alpha\de_k)$,
{ lexicographically ordered}.

We start by fixing a point $x_0\in B_0$ for which there are many good choices in each neighboring ball. We next select
$x_1\in B_1$ among the points which are good choices for $x_0$ and which have many good choices in each neighboring
subsequent ball $B_i$, $i\geq 2$. Iterating the process, the point $x_m\in B_m$ will be taken among the good choices
for the neighboring previously fixed $x_i$, $i<m$, and with the property that have many good choices in the neighboring
subsequent $B_i$, $i>m$. Since each ball can have at most seven neighbors, at each step we select $x_m$ avoiding just 
a small subset of $B_m$.

We call $S$ the set of points obtained by this process and we construct a new triangulation,
with $x,y$ neighbors if and only if $\bar x$, $\bar y$ are neighbors.
Notice that again 
\be{\label{eq:nondeg}c_1\de_k\leq|x-y|\leq c_2\de_k,} 
for every couple of neighboring points $x,y$, with the same $k$
as for the corresponding reference points
$\overline x$ and $\overline y$, and suitable $c_1,c_2>0$ independent from $k$.
We finally define $\phi(u)$ as the linear interpolation 
between the values of  $u(x)$, $x\in S$ on each triangle of the triangulation. 
 
\medskip

Fixed a triangle $T$ and any couple of its vertices $x,y$, we compute a component of the constant matrix $e(\phi(u))$ on $T$ by 
\be{\label{eq:eu}e(\phi(u))\nu\cdot  \nu=\frac{(\phi(u)(x)-\phi(u)(y))\cdot \nu}{|x-y|}=\mint_{ s_{x,y} }(u^\nu_z)'dt,}
where $\nu$ and $z$ are defined in \eqref{eq:n} and \eqref{eq:z}. We used the fact that $u$ and $\phi(u)$ agree on $x$ and $y$ 
and that $u$ is $W^{1,1}( s_{x,y} )$ by the choice of $x$ and $y$. By \eqref{eq:eu}, \eqref{eq:nondeg}, and property \ref{tecp3} above it follows
$$|e(\phi(u))\nu\cdot \nu|\leq \frac{\tilde{c}}{\de_k^2}\int_{C}|e(u)|dx',$$
where $C$ is defined in \eqref{eq:C}. We recall that here and henceforth $\tilde{c}$ can possibly change. Letting $\nu$ vary among the directions of the sides of $T$, we obtain a control on 
the whole $|e(\phi(u))|$ thanks to \eqref{eq:nondeg}
\be{\label{eq:euconv}|e(\phi(u))|\leq \frac{\tilde{c}}{\de_k^2}\int_{C_T}|e(u)|dx',}
where $C_T$ denotes the convex envelope  
\begin{equation*}
C_T:=\conv(\cup B(\overline{x},\alpha\de_k)) 
\end{equation*}
and the union is taken over the three vertices $\overline{x}$ in the old triangulation corresponding
to the three vertices of $T$. We remark that 
$B(\overline x,\alpha\delta_k)\subset B_{R_{k+1}}\setminus B_{R_{k-1}}$ for all $\overline x\in \partial B_{R_k}$, therefore 
the sets $C_T$ have finite overlap.

We are ready to prove property \ref{tecnico2}. By Jensen's inequality and \eqref{eq:euconv} we have for $1\leq q\leq p$
\begin{align*}
\int_T |e(\phi(u))|^qdx'&= \calL^2(T)|e(\phi(u))|^q\leq \tilde{c}\calL^2(C_T) \Big(\mint_{C_T}|e(u)|dx'\Big)^q\nonumber\\
&\leq \tilde{c}\int_{C_T}|e(u)|^qdx',
\end{align*}
and finally summing up on all triangles $T$ we get the conclusion. 

In order to prove properties \ref{tecnico3} and \ref{tecnico4} we estimate 
\be{\label{eq:tr1}\int_{T} |u-\phi(u)|dx'\leq \int_{T} |u-a_{\overline{x},\overline{y}}|dx' + \int_{T}|a_{\overline{x},\overline{y}}-\phi(u)|dx',}
where $T$ is again a triangle of the modified triangulation
with vertices $x,y,z$, 
$\overline{x},\overline{y},\overline{z}$ denote the three
corresponding vertices  of the old triangulation,
$a_{\overline{x},\overline{y}}$ is the infinitesimal rigid motion appearing in the Poincar\'e's inequality for $u$ on $Q_{\overline{x},\overline{y}}$ (see item \ref{tecp4} above).

Let us study first the second term in \eqref{eq:tr1}. Since $a_{\overline{x},\overline{y}}-\phi(u)$ is affine, it achieves its maximum on a vertex $\xi$ of $T$, therefore  
\begin{equation*}
\int_{T}|a_{\overline{x},\overline{y}}-\phi(u)|dx'\leq c\de_k^2|a_{\overline{x},\overline{y}}(\xi)-\phi(u)(\xi)|
=c\de_k^2|a_{\overline{x},\overline{y}}(\xi)-u(\xi)|. 
\end{equation*}
Notice that if $\xi=\overline{z}$ then
\ba{\label{eq:tr4}c\de_k^2|a_{\overline{x},\overline{y}}(\xi)-u(\xi)|&\leq& c\de_k^2|a_{\overline{x},\overline{y}}(\xi)-a_{\overline{x},\xi}(\xi)|+c\de_k^2|a_{\overline{x},\xi}(\xi)-u(\xi)|
\nonumber\\
&\leq& c\int_{B(\overline{x},\alpha\delta_k)}|a_{\overline{x},\overline{y}}-a_{\overline{x},\xi}|dx'+c\de_k |Eu|(Q_{\overline{x},\xi}),}
where we used the fact that $a_{\overline{x},\xi},a_{\overline{x},\overline{y}}$ are affine and item \ref{tecp4} above;
if $\xi\in\{\overline x, \overline y\}$ then only the second term appears.

By \eqref{eq:tr1}-\eqref{eq:tr4}, the triangular inequality, and Poincar\'e's inequality we conclude
\be{\label{eq:tr5}\int_{T} |u-\phi(u)|dx'\leq c\delta_k |Eu|(Q_{T}),}
where $Q_T:=Q_{\overline{x},\overline{y}}\cup Q_{\overline{y},\overline z}\cup Q_{\overline{z},\overline x}$.
Finally summing up over $T$ we obtain property \ref{tecnico3}. 

We prove now property \ref{tecnico4}, property \ref{tecnico5} holding true by construction.
We define $\phi(u):=u$ outside $\overline B_R$ and know that $\phi(u)\in W^{1,p}(B_R,\R^2)\cap SBD(B_{2r})$.
It remains to prove that the traces on $\partial B_R$
coincide, or, equivalently, that $\calH^1(J_{\phi(u)}\cap B_R)=0$.
Let $\psi_k\in C^\infty(B_R)$ be such that $\psi_k=0$
on $B_{R_k}$, $\psi_k=1$ in a neighborhood of $\partial B_R$, and $|\nabla\psi_k|\leq c/\de_k$. We define 
$v_k:=(u-\phi(u))\psi_k\in SBD(B_R)$ and we prove that $v_k\to 0$ strongly in $BD$, this implying in turn
$v_k|_{\partial B_R}\to 0$ in $L^1(\partial B_R,\R^2)$ in the sense of traces and therefore property  \ref{tecnico4}. Clearly
\begin{equation*}
\int_{B_R}|v_k|dx\leq\int_{B_R\setminus B_{R_k}}|u-\phi(u)|dx\to0 
\end{equation*}
by the dominated convergence theorem. Finally, using \eqref{eq:tr5} and the fact that the triangles have finite overlap,
\bes{|Ev_k|(B_R)&\leq&|E(u-\phi(u))|(B_R\setminus B_{R_k})+\frac{c}{\de_k}\int_{B_R\setminus B_{R_k}}|u-\phi(u)|dx\\
&\leq&\tilde{c}|E(u-\phi(u))|(B_R\setminus B_{R_k}),
}
the last term tends to $0$ and this concludes the proof of property \ref{tecnico4}.

\end{proof}

\section{Regularity of \texorpdfstring{$SBD^p$}{SBDp} functions with small jump set}
\label{sec:approximation}

We first discuss how $SBD^p$ functions can be approximated by $W^{1,p}$ functions locally away from the 
jump set (Section \ref{sec:approxbulk}), and then how they can be approximated by piecewise $W^{1,p}$ functions around the jump set
(Section \ref{sec:approxinterface}).
Our approximation result also leads to the Korn inequality stated in Theorem \ref{theo:korn}.
The key ingredient for all these results is the  construction of Theorem \ref{t:tecnico}.
Throughout the section $\eta\in(0,1)$ will be the constant from Theorem~\ref{t:tecnico} and $n=2$.

\subsection{Approximation of \texorpdfstring{$SBD^p$}{SBDp} functions with \texorpdfstring{$W^{1,p}$}{W1p} functions}
\label{sec:approxbulk}
We shall use that the construction of Theorem \ref{t:tecnico}, using a suitable covering argument,
permits to 
approximate $SBD^p$ functions by $W^{1,p}$ functions which coincide away from a small neighborhood of the jump set.
The neighborhood is the union of countably many balls, such that each of them contains an amount of jump set
proportional to the radius. Before discussing the covering argument
in Proposition \ref{p:ricopr}, we show that (away from the boundary) almost any point of 
the jump set is the center of a ball with the appropriate
density.

\begin{lemma}\label{l:rx}
Let 
 $s\in (0,1)$.
Let  $J\in\mathcal{B}(B_{2\rho})$, for some $\rho>0$, be such that $\Huno(J)<\eta(1-s)\rho$
then for $\Huno$-a.e. $x\in J\cap B_{2s\rho}$ there exists a radius $r_x\in(0,(1-s)\rho)$ such that
\begin{equation}\label{e:rx0}
\Huno\big(J\cap \partial B_{r_x}(x)\big)=0,
\end{equation}
\begin{equation}\label{e:rx}
\eta\, {r_x}\leq\Huno\big(J\cap B_{r_x}(x)\big)\leq
\Huno\big(J\cap B_{2r_x}(x)\big)<{2\,\eta\, r_x}.
\end{equation}
\end{lemma}
\begin{proof}
{ 
We fix $x\in J\cap B_{2s\rho}$, choose $\lambda_x\in(\rho ,2\rho)$ such that }
$\Huno\big(J\cap \partial B_{\sfrac{\lambda_x}{2^{k}}}(x)\big)=0$ for all $k\in\N$, and define 
\[
r_x:=\max\{\sfrac{\lambda_x}{2^{k}}:\,k\in\N,\,
\Huno\big(J\cap B_{\sfrac{\lambda_x}{2^{k}}}(x)\big)\geq {{\eta\,\lambda_x}{2^{-k}}}\}.
\]
{The set is nonempty for $\calH^1$-almost every $x$ because $\eta<1$.
The estimates 
\eqref{e:rx} hold
 by definition.}
To conclude that $r_x<(1-s)\rho$ it is enough to notice that the opposite inequality would give the ensuing contradiction
\[
\Huno(J)\geq \Huno\big(J\cap B_{r_x}(x)\big)\geq \eta\,r_x 
\geq{(1-s)}\eta\,\rho>\Huno(J).
\]
\end{proof}

We are now ready to prove the main result of the section via a covering argument, Lemma~\ref{l:rx}, and 
Theorem~\ref{t:tecnico}.

\begin{proposition}\label{p:ricopr}
Let $p\in(1,\infty)$, $n=2$.
There exists a universal constant $c>0$ such that if $u\in SBD^p(B_{2\rho})$, {$\rho>0$,} satisfies
\[
\Huno(J_u\cap B_{2\rho})<{\eta\,(1-s)\rho}
\]
for $\eta\in(0,1)$ as in Theorem~\ref{t:tecnico} and  some $s\in(0,1)$,
then there is a countable family $\calF=\{B\}$ of closed balls of radius $r_B<(1-s)\rho$,
each contained in $B_{(1+s)\rho}$, and a field $w\in SBD^p(B_{2\rho})$ 
such that
\begin{enumerate}
\item $\rho^{-1}\sum_{\calF}{\Ln}\big(B\big)+\sum_{\calF}\Huno\big(\partial B\big)
\leq{\sfrac {c}\eta}\,\Huno(J_u\cap B_{2\rho})$;
\item $\Huno\big(J_u\cap\cup_{\calF}\partial B\big)=\Huno\big((J_u\cap {B_{2s\rho}})\setminus \cup_{\calF}B\big)=0$;
\item $w= u$ {$\Ln$-a.e.} on $B_{2\rho}\setminus\cup_{\calF}B$;
\item {$w\in W^{1,p}(B_{2s\rho},\Rdue)$} and $\Huno(J_w\setminus J_u)=0$;
\item 
\begin{equation}\label{e:volume}
 \int_{\cup_{\calF}B}|e(w)|^pdx\leq c\int_{\cup_{\calF}B}|e(u)|^pdx; 
\end{equation}
and there exists a skew-symmetric matrix $A$ such that 
\begin{equation}\label{e:korn}
 \int_{B_{2s\rho}\setminus\cup_{\calF}B}|\nabla u-A|^pdx\leq c(p)\int_{B_{2\rho}}|e(u)|^pdx; 
\end{equation}
\item $\|u-w\|_{L^1(B,\Rdue)}\leq c\,r_B\,|Eu|(B)$, for every $B\in \calF$;
\item if, additionally, $u\in L^\infty(B_{2\rho},\Rdue)$ then $w\in L^\infty(B_{2\rho},\Rdue)$ with
\[
\|w\|_{L^\infty(B_{2\rho},\Rdue)}\leq {c}\|u\|_{L^\infty(B_{2\rho},\Rdue)}.
\]
\end{enumerate}
\end{proposition}
\begin{proof}
By Lemma~\ref{l:rx} we find a family $\calF^\prime$ of open balls covering $\Huno$-a.e. $J_u\cap B_{{2s\rho}}$ 
that satisfies \eqref{e:rx0} and \eqref{e:rx}. 
{ Setting $J=J_u$, to every $B\in\calF^\prime$ we associate
a new ball $B^*\subset B$ with the properties \ref{tecnico1}-\ref{tecnico5} of 
Theorem~\ref{t:tecnico}. Let $\calF^*$ be the family of the new balls $B^*$,
this is still a cover of $J$. Further, the balls $B^*$ can be taken to be closed.
By the Besicovitch covering theorem \cite[Theorem 2.17]{ambrosio} there are $\xi$} countable subfamilies 
$\calF^\prime_j=\{B_j^i\}_{i\in\N}$ of disjoint balls. Therefore, setting $\calF:=\cup_{j=1}^\xi \calF^\prime_j$ 
we have $\Huno\big((J_u\cap B_{2s\rho})\setminus \cup_{\calF}B\big)=0$. In addition, by \eqref{e:rx0} 
the first condition in item (ii) is satisfied as well, so that (ii) is established. 
Furthermore, 
\begin{align*}
\sum_{B\in\calF}\Huno(\partial B)= &2\pi\sum_{B\in\calF}{r_B}
\stackrel{\eqref{e:rx}}{\leq} \frac{2\pi}{\eta}\sum_{B\in\calF}\Huno\big(J_u\cap B\big)\\
\leq&\xi\,\frac{2\pi}{\eta}\Huno\big(J_u\cap \cup_{B\in\calF}B\big)\leq
\xi\,{\frac{2\pi}{\eta}}\Huno(J_u\cap B_{2\rho}).
\end{align*}
The volume estimate follows since $r_B\le \rho$ implies $\sum r_B^2\le \rho\sum r_B$. We remark that a 
quadratic volume estimate also follows by
$\sum r_B^2\le (\sum r_B)^2$.

Let $\phi(u)$ be the function given by Theorem \ref{t:tecnico} on the balls of the first family $\calF^\prime_1$ and define for every $h\in\N$ a function
\[
w_1^h:=\begin{cases}
\phi(u) & B_1^i, \,i\leq h \cr
u &  \mbox{otherwise}
\end{cases}
\]
such that $w_1^h\in SBD^p(B_{2\rho})$, $w_1^h\in W^{1,p}(\cup_{i\leq h}B_1^i;\Rdue)$ 
with $w_1^h=u$ $\Ln$-a.e. on $B_{2\rho}\setminus \cup_{i\leq h}B_1^i$ and $\Huno(J_{w_1^h}\setminus J_u)=0$.
In addition by item  \ref{tecnico2} in Theorem~\ref{t:tecnico}
\begin{multline}\label{e:stimavolwh}
\int_{B_{2\rho}}|e(w_1^h)|^p\,dx=\int_{\cup_{i\leq h}B_1^i}|e(\phi(u))|^p\,dx+
\int_{B_{2\rho}\setminus \cup_{i\leq h}B_1^i}|e(u)|^p\,dx\\
\leq {\tilde{c}}\int_{\cup_{i\leq h}B_1^i}|e(u)|^p\,dx+
\int_{B_{2\rho}\setminus \cup_{i\leq h}B_1^i}|e(u)|^p\,dx\leq(1+{\tilde{c}})\int_{B_{2\rho}}|e(u)|^p\,dx,
\end{multline}
and
\[
 |Ew_1^h|(B_{2\rho})\leq |Eu|\big(B_{2\rho}\setminus \cup_{i\leq h}B_1^i\big)
 +{\tilde{c}}\int_{\cup_{i\leq h}B_1^i}|e(u)|\,dx.
\]
Moreover, recalling that the $B_1^i$'s are disjoint and that $w_1^{h-1}=u$ on $B_1^h$, item  \ref{tecnico3} in 
Theorem~\ref{t:tecnico} gives
\[
 \|w_1^h-w_1^{h-1}\|_{L^1(B_{2\rho};\Rdue)}= \|w_1^h-u\|_{L^1(B_1^h;\Rdue)}\leq c\,\rho\,|Eu|(B_1^h),
\]
in turn implying that for all $h\geq k\geq 1$ 
\[
 \|w_1^h-w_1^k\|_{L^1(B_{2\rho};\Rdue)}\leq \sum_{i=k+1}^h\|w_1^i-w_1^{i-1}\|_{L^1(B_1^h;\Rdue)}
 \leq c\,\rho\,|Eu|\big(\cup_{k+1\leq i\leq h}B_1^i\big).
\]
Thus, $w_1^h\to w_1$ in $L^1(B_{2\rho};\Rdue)$ with
\[
w_1:=\begin{cases}
\phi(u) & \cup_{\calF_1^\prime}B \cr
u & \mbox{otherwise}.
\end{cases}
\]
The $BD$ compactness theorem then yields that $w_1\in BD(B_{2\rho})$. In turn, by \eqref{e:stimavolwh} 
and since $\Huno( J_{w_1^h}\setminus J_u)=0$, the $SBD$ compactness theorem implies that actually 
$w_1\in SBD^p(B_{2\rho})$ {(see also \cite[Theorem 11.3]{gbd}).}
Furthermore, since 
\[
 \Huno\big(J_{w_1^h}\cap\cup_{\calF_1^\prime}B\big)=\Huno(J_u\cap \cup_{i\geq h+1} B_1^i\big),
\]
we may conclude that 
\[
 \Huno\big(J_{w_1}\cap\cup_{\calF_1^\prime}B\big)\leq
 \liminf_h\Huno\big(J_{w_1^h}\cap\cup_{\calF_1^\prime}B\big)=0,
\]
and therefore $w_1\in W^{1,p}(\cup_{\calF_1^\prime}B,\Rdue)$. 
Finally, by construction $w_1=u$ $\Ln$-a.e. on $B_{2\rho}\setminus \cup_{\calF_1^\prime}B$ and 
$\Huno(J_{w_1}\setminus J_u)=0$.

By iterating the latter construction, for all $1< k\leq\xi$ and for every $h\in\N$ we find 
\[
w_k^h:=\begin{cases}
\phi(w_{k-1}) & B_k^i, \,i\leq h \cr
w_{k-1} & \mbox{otherwise}
\end{cases}
\]
such that $w_k^h\in SBD^p(B_{2\rho})$, $w_k^h\in W^{1,p}(\cup_{i\leq h}B_k^i;\Rdue)$,
$w_k^h=w_{k-1}$ $\Ln$-a.e. on $B_{2\rho}\setminus \cup_{i\leq h}B_k^i$ , 
$\Huno(J_{w_k^h}\setminus J_{w_{k-1}})=0$.
In addition, arguing as above, $w_k^h\to w_k$ in $L^1(B_{2\rho},\Rdue)$ with
\[
w_k:=\begin{cases}
\phi(w_{k-1}) & \cup_{\calF_k^\prime}B \cr
w_{k-1} & \mbox{otherwise},
\end{cases}
\]
$w_k\in SBD^p(B_{2\rho})$, $w_k\in W^{1,p}(\cup_{j\leq k}\cup_{\calF_j^\prime}B;\Rdue)$, 
$w_k=w_{k-1}$ $\Ln$-a.e. on $B_{2\rho}\setminus \cup_{\calF_k^\prime}B$ and 
$\Huno(J_{w_k}\setminus J_{w_{k-1}})=0$.

Set $w:=w_{\xi}$, then $w\in SBD^p(B_{2\rho})$, $w\in W^{1,p}(\cup_{\calF}B;\Rdue)$, 
$w=u$ $\Ln$-a.e. on $B_{2\rho}\setminus \cup_{\calF}B$, $\Huno(J_w\setminus J_u)=0$.
Iterating estimate \eqref{e:stimavolwh}, inequality \eqref{e:volume} follows at once with 
$c:=\max\{1+\tilde{c},2\pi\}\,\xi$, with $\tilde{c}$ the constant in Theorem~\ref{t:tecnico}. 
{Korn's inequality \eqref{e:korn} follows now immediately by (iii), (iv), and \eqref{e:volume}.} 

Finally, it is clear that also items (vi) and (vii) are satisfied in view of properties  \ref{tecnico3} and  \ref{tecnico5} 
in Theorem~\ref{t:tecnico}.
\end{proof}
\subsection{Korn's inequality in \texorpdfstring{$SBD^p$}{SBDp}}
\label{sec:korn}
\begin{proof}[Proof of Theorem \ref{theo:korn}]
By standard scaling and covering arguments it suffices to prove the assertion for a special Lipschitz
domain. Precisely, let $\varphi:\R\to\R$ Lipschitz with $\min \varphi[(-1,1)]=2$, and set
$U:=\{x: x_1\in (-2,2), x_2\in (-2,\varphi(x_1))\}$,
and $U^\mathrm{int}:=\{x: x_1\in (-1,1), x_2\in (-1,\varphi(x_1))\}$.
It suffices to show that for any $u\in SBD^p(U)$ there are $\omega$ with $\calH^1(\partial\omega)
+|\omega|^{1/2}\le
c \calH^1(J_u)$ 
and an affine function $a:\R^2\to\R^2$ such that $\|u-a\|_{L^p(U^\mathrm{int}\setminus\omega,\R^2)} 
+\|\nabla u-\nabla a\|_{L^p(U^\mathrm{int}\setminus\omega,\R^2)} 
\le c_{L,p} \|e(u)\|_{L^p(U,\R^{2\times 2})}$, with $c$ depending on $p$ and the Lipschitz constant $L$ of $\varphi$.
Obviously we can assume $\calH^1(J_u)$ to be small.

Let $q_j:=x_j+(-r_j/2,r_j/2)^2$ 
and $Q_j:=x_j+(-r_j,r_j)^2$, and assume that $u\in SBD^p(Q_j)$ obeys $\calH^1(J_u\cap Q_j)\le \eta r_j/8$.
By Proposition \ref{p:ricopr} with $\rho:=r_j/2$ and $s:=1/\sqrt2$ and Poincar\'e's inequality 
there are $\omega_j$ and $a_j$ affine with $r_j\calH^1(\partial\omega_j)+|\omega_j|^{1/2}\le c\calH^1(J_u\cap Q_j)$ 
and $r_j^{-1}\|u_j-a_j\|_{L^p(q_j\setminus\omega_j,\R^2)}+
\|\nabla u_j-\nabla a_j\|_{L^p(q_j\setminus\omega_j,\R^{2\times 2})}\le c_p \|e(u)\|_{L^p(Q_j,\R^{2\times 2})}$, with a constant
which depends only on the exponent $p$.

To pass to the estimate on $U^\mathrm{int}$ one uses a Whitney cover with pairs of open cubes
 $q_j$ and $Q_j$ such that the exterior ones have finite overlap and the interior ones form a cover, 
 as done for example in proving the nonlinear Korn's inequality in \cite[Theorem 3.1]{FrieseckeJamesMueller2002}.
Following \cite{friedrich},
if $\calH^1(J_u\cap Q_j)\ge \eta r_j/8$ we define $P_j:=x_j+(-r_j,r_j)\times(-r_j,\infty) \cap U$,
otherwise $P_j=\emptyset$ and $\omega_j$, $a_j$ are obtained as above.
Notice that $\calH^1(P_j)\le c_L r_j$ by the properties of Lipschitz functions and of the Whitney covering.

Then it suffices to apply the weighted Poincar\'e inequality, as done in \cite[Theorem 3.1]{FrieseckeJamesMueller2002}
and  \cite[Theorem 4.2]{friedrich}. By the properties of the covering, for neighboring pairs of squares
$|q_j\cap q_i|\ge c r_i^2$, and if $\eta$ is not too large this gives a bound on the difference of the two affine functions.
One then defines $a^*\in C^\infty(U,\R^{2})$ 
using a partition of unit subordinated to the cover $\{q_j\}$, and obtains 
\begin{equation*}
 \int_{U^\mathrm{int}\setminus \cup_j P_j} (\varphi(x_1)-x_2)^p |D^2a^*|^p(x) dx\le c_{L,p} \|e(u)\|_{L^p(U,\R^2)}^p\,.
\end{equation*}
Since the cube $Q_0=(-2,2)^2$ was not removed one has $a^*=a_0$ in $q_0$ and application of the one-dimensional weighted Poincar\'e
inequality to $a^*(x_1,\cdot)$ leads to the assertion, with $\omega:=\cup_j (P_j\cup \omega_j)$ and $a:=a_0$. Equivalently, in the last step one may use
a Poincar\'e or Korn inequality on John domains, as done in  \cite[Theorem 4.2]{friedrich}.
\end{proof}
We remark that the nonoptimality of the exponent in \cite[Theorem 4.2]{friedrich} is only consequence
of the nonoptimal local estimate employed there (see \cite[Theorem 3.1]{friedrich}).

\subsection{{Reflection}}
\label{sec:approxinterface}
{In this subsection we establish a technical result instrumental for the identification 
of the surface energy density in Section~\ref{s:surface}. To this aim, given $u\in SBD^p(\Omega)$ and a 
point $x_0\in J_u$ we set 
\begin{equation}\label{e:uxzero}
 u_{x_0}(x):=\begin{cases}
              u^+(x_0) & \text{ if } \langle x-x_0,\nu_{x_0}\rangle >0,\\
              u^-(x_0) & \text{ if } \langle x-x_0,\nu_{x_0}\rangle <0.
             \end{cases}
\end{equation}
}
\begin{lemma}\label{lem:refl}
{Let $p\in(1,\infty)$}, $u\in SBD^p(\Omega)$, $\Omega\subset\R^2$ open. For $\calH^1$-a.e. $x_0\in J_u$ and any $\rho>0$ 
  sufficiently small there is $v_\rho\in SBD^p(B_{2\rho}(x_0))\cap SBV^p{(B_{\rho}(x_0),\R^2)}$ such that:
  \begin{enumerate}
   \item $\displaystyle\lim_{\rho\to0} \frac1\rho \calH^1(J_{v_\rho}\setminus J_u \cap B_{\rho}(x_0)) =0$;
   \item $\displaystyle\lim_{\rho\to0} {\frac1\rho \int_{B_{\rho}(x_0)} |\nabla v_\rho|^p dx =0}$;
   \item $\displaystyle\lim_{\rho\to0} \frac1{\rho^2} \calL^2(\{x\in B_{\rho}(x_0):\,u\ne v_\rho\}) =0$;
   \item $\displaystyle\lim_{\rho\to0} \frac1{\rho^2} \int_{B_{\rho}(x_0)} |v_\rho-u| dx =0$;
   \item $\displaystyle\lim_{\rho\to0} \frac1{\rho^{p+1}} \int_{B_{\rho}(x_0)} |v_\rho-u_{x_0}|^p dx =0$; 
   \item $\displaystyle\lim_{\rho\to0} \frac1{\rho} \int_{B_{\rho}(x_0)\cap{J_{v_\rho}}} |[v_\rho]-[u]| d\calH^1 =0$.
     \end{enumerate}
\end{lemma}
\begin{proof}
Since $J_u$ is $(\calH^1,1)$ rectifiable, there exists a sequence $(\Gamma_i)_{i=1}^\infty$ of $C^1$ curves such that $\calH^1(J_u\setminus \cup_{i=1}^\infty\Gamma_i)=0$.
For $\calH^1$-a.e. $x_0\in J_u$ we have
\bes{
&\displaystyle{\lim_{\rho\to 0}\frac1{2\rho}{\int_{J_u\cap B_{\rho}(x_0)}(|[u]|+1)d\calH^1}=|[u](x_0)|+1},\\
&\displaystyle{\lim_{\rho\to 0}\frac1{2\rho}{\int_{J_u\cap\Gamma\cap B_{\rho}(x_0)}(|[u]|+1)d\calH^1}=|[u](x_0)|+1},}
for one of the aforementioned curves $\Gamma$.
Therefore
\be{\label{eq:diffsimm}\displaystyle{\lim_{\rho\to 0}\frac1{2\rho}}{\int_{(J_u\triangle\Gamma)\cap B_{\rho}(x_0)}(|[u]|+1)d\calH^1}=0}
and for $\rho$ small $\Gamma$ separates $B_{6\rho}(x_0)$ into two connected components. 
It is not restrictive to assume  that {$\Gamma\cap B_{6\rho}(x_0)$} is the graph of 
a function $h\in C^1(\R)$.
Moreover the following properties hold $\calH^1$-a.e. $x_0\in J_u$ 
\ba{
&\displaystyle\lim_{\rho\to0} \frac1\rho \int_{B_\rho(x_0)} |e(u)|^p dx =0,\label{e:choice1}\\
&\displaystyle\lim_{\rho\to0} \frac1{2\rho}|Eu|(B_\rho(x_0))=|[u]\odot\nu_u|(x_0),\label{e:choice2}\\
&\displaystyle\lim_{\rho\to0} \frac1{\rho^2} \int_{B_\rho(x_0)\cap\{\pm(x_2-h(x_1))>0\}}|u-u^\pm(x_0)|dx=0.\label{e:choice3}
}

For simplicity we next assume that the point $x_0=0$ satisfies all the previous properties \eqref{eq:diffsimm}-\eqref{e:choice3}, 
{with $h(0)=h^\prime(0)=0$}. We also set 
$\tau_\rho:=\|h\|_{L^\infty(B_{6\rho})}$ and note that $\tau_\rho/\rho\to 0$ as $\rho\to0$. We now define the reflections of $u$ with respect to the lines $\{x_2=\pm\tau_\rho\}$,
in the sense of \cite[Lemma 1]{nitsche}. {More precisely, define $\tilde u^+_\rho$ 
on the set ${B_{2\rho}}\cap\{x_2<\tau_\rho\}$ by}
\[
\begin{cases}
(\tilde u^+_\rho)_1(x_1,x_2):= -2 u_1(x_1,3\tau_\rho-2x_2)+3 u_1(x_1,2\tau_\rho-x_2)\\
(\tilde u^+_\rho)_2(x_1,x_2):=  4 u_2(x_1,3\tau_\rho-2x_2)-3 u_2(x_1,2\tau_\rho-x_2)
\end{cases}
\]
and by $u$ otherwise in ${B_{2\rho}}$. Note that $\tilde u^+_\rho\in SBD^p({B_{2\rho}})$ and that 
\ba{& {\displaystyle\lim_{\rho\to0}\frac1{2\rho}\calH^1(J_{\tilde u^+_\rho}\cap B_{2\rho})=0,}\label{eq:Jutilde}\\
&\|e(\tilde u^+_\rho)\|_{L^p(B_{2\rho},\R^{2{\times}2})}\leq c \|e(u)\|_{L^p(B_{6\rho},\R^{2{\times}2})},\label{eq:eutilde}} 
for a universal constant $c$.
Using a similar reflection we define $\tilde u^-_\rho$ in $B_{2\rho}\cap\{(x_1,x_2):x_2>-\tau_\rho\}$ and we set $\tilde u^-_\rho:=u$ otherwise in $B_{2\rho}$.

By \eqref{eq:diffsimm} and \eqref{eq:Jutilde} for $\rho$ small we have that $\tilde u^\pm_\rho$ satisfy the 
hypotheses of Proposition \ref{p:ricopr} on $B_{2\rho}$ with $s=1/2$. 
Thus, there exist $w^\pm_\rho\in SBD^p(B_{2\rho})\cap W^{1,p}(B_\rho,\R^2)$, 
for which properties (i)-(vii) hold true. Finally let us define $v_{\rho}\in SBD^p(B_{2\rho})$ by
\[
 v_{\rho}:=\begin{cases}
        w^+_\rho & \textrm{in $B_{2\rho}\cap\{x_2>h(x_1)\}$,}\\
        w^-_\rho & \textrm{in $B_{2\rho}\cap\{x_2<h(x_1)\}$.}
        \end{cases}
\]
Since 
$w^\pm_\rho\in W^{1,p}(B_{\rho},\R^2)$ we obtain $v_\rho\in SBV^p(B_\rho,\R^2)$ with
\begin{align*}
Dv_\rho\res B_{\rho}=&
         \nabla w^+_\rho\,\Ln\res B_{\rho}\cap\{x_2>h(x_1)\} +
         (w^+_\rho-w^-_\rho)\otimes\nu_\Gamma\Huno\res\Gamma\cap B_{\rho}\\
         &+
         \nabla w^-_\rho\,\Ln\res B_{\rho}\cap\{x_2<h(x_1)\}.
\end{align*}
We next check that $v_\rho$ satisfies the properties in the statement in the ball $B_\rho$. 
Property (i) comes straightforwardly from \eqref{eq:diffsimm} and from the fact that 
$J_{v_\rho}\subset \Gamma$. Moreover \eqref{eq:eutilde}, \eqref{e:volume}, and \eqref{e:choice1} yield 
\begin{equation}\label{e:evrho}
 \lim_{\rho\to 0} \frac{1}{\rho}\int_{B_{\rho}}|e(w_\rho^\pm)|^pdx=0.
\end{equation} 
As for property (iii), we observe that
\[
\lim_{\rho\to 0}\frac1{\rho^2} \calL^2(\{B_\rho:\,u\ne v_\rho\})\leq 
\lim_{\rho\to 0}(c\frac{\tau_\rho}{\rho}+
\frac c\rho\calH^1((J_u\setminus \Gamma)\cap B_{6\rho}))=0,
\]
where we have used Proposition~\ref{p:ricopr} (i) and \eqref{eq:diffsimm}.

Let us now prove property (iv). By the definition of $v_{\rho}$ and $\tilde u^\pm_\rho$ and by triangular inequality 
we obtain
\bm{\label{e:tr}
\frac1{\rho^2} \int_{B_{\rho}}|v_\rho-u| dx \leq\\
\frac1{\rho^2} \int_{B_{\rho}\cap \{h(x_1)<x_2\}}|w^+_\rho-\tilde u^+_\rho| dx + 
\frac1{\rho^2} \int_{B_{\rho}\cap \{h(x_1)<x_2<\tau_\rho\}}|\tilde u^+_\rho-u| dx+\\
\frac1{\rho^2} \int_{B_{\rho}\cap\{x_2<h(x_1)\}}|w^-_\rho-\tilde u^-_\rho| dx +
\frac1{\rho^2} \int_{B_{\rho}\cap\{-\tau_\rho<x_2<h(x_1)\}}|\tilde u^-_\rho-u| dx.
}
By the definition of $w^+_\rho$ and Proposition \ref{p:ricopr} (vi) we can estimate
\begin{equation*}
\frac1{\rho^2} \int_{B_{\rho}}|w^+_\rho-\tilde u^+_\rho| dx\leq
{\frac {c}{\rho} |E\tilde u_\rho^+|(B_{2\rho}) \le \frac{c}{\rho} |Eu|(B_{6\rho}\setminus\Gamma).
}
\end{equation*}
By \eqref{eq:diffsimm} and (\ref{e:choice1}) we conclude 
that the first term of \eqref{e:tr} tends to $0$. {Clearly, the same argument can be applied to the 
third term there. So, it remains to treat the second term in \eqref{e:tr}, being the fourth one similar.} 
By triangular inequality and a change of variable we infer
\bm{
\frac1{\rho^2}\int_{{B_{\rho}}\cap \{h(x_1)<x_2<\tau_\rho\}}|\tilde u^+_\rho-u| dx\leq\nonumber\\
\frac1{\rho^2}\int_{B_{\rho}}|\tilde u^+_\rho-{u^+(x_0)}| dx+
\frac1{\rho^2}\int_{B_{\rho}\cap \{h(x_1)<x_2\}}|{u^+(x_0)}-u| dx\leq\\
\frac c{\rho^2}\int_{B_{6\rho}\cap \{h(x_1)<x_2\}}|{u^+(x_0)}-u| dx,
}
and the last term tends to $0$ by \eqref{e:choice3}, hence property (iv) follows.

Let us prove now property (v). 
By Korn's inequality and Poincar\'e's inequality in $W^{1,p}$, 
there exists an affine function
$a_\rho(x):=d_\rho+\beta_\rho x$  such that
\begin{equation}\label{e:ppoi}
 \frac1{\rho^{p+1}} \int_{B_{\rho}} |w^+_\rho-a_\rho|^p dx \leq \frac c{\rho}\int_{B_{\rho}}|e(w^+_\rho)|^pdx.
\end{equation}
{We first claim that 
\begin{equation}\label{e:drho}
 \lim_{\rho\to0}d_{\rho}=u^+(x_0).
\end{equation}
{Let $\omega^+_\rho:=B_\rho\cap \{u= w_\rho^+\}\cap\{x_2>h(x_1)\}$. 
Since $|\omega^+_\rho|/\rho^2\to \pi/2$, and $a_\rho$ is affine,
by \cite[Lemma 4.3]{ContiFocardiIurlano2016} we obtain, for $\rho$ small,}
\[
\|a_\rho-u^+(x_0)\|_{L^\infty({B^+_{\rho}},\R^2)}
\leq \frac c{\rho^2}\int_{\omega_\rho^+}|w^+_\rho-a_\rho|dx+\frac c{\rho^2}\int_{\omega_\rho^+}|u-u^+(x_0)|dx.
\] 
The right hand side above is infinitesimal by \eqref{e:ppoi}, {(\ref{e:evrho})} and \eqref{e:choice3}, thus we conclude
\[
 \limsup_{\rho\to0}|d_\rho-u^+(x_0)|\leq
 \lim_{\rho\to0}\|a_\rho-u^+(x_0)\|_{L^\infty(B^+_{\rho},\R^2)}=0,
\]
which proves (\ref{e:drho}).

Next we prove that}
\ba{\label{e:betax0}
&\displaystyle{\lim_{\rho\to0}\rho|\beta_\rho|^p=0,}\\
&\displaystyle{{\lim_{\rho\to0}{\rho^{\frac{1-p}{p}}}{|d_\rho- u^+(x_0)|}=0}}.\label{e:ax0}
}
To establish \eqref{e:betax0}, we fix $\delta>0$ small and we consider $\hat{\rho}$ such that
\be{\label{e:delta}
\Big(\frac1\rho\int_{B_{\rho}}|e(w^+_\rho)|^pdx\Big)^{\frac1p}<\delta, \qquad\textrm{for $\rho\leq\hat{\rho}$,}
}
note that this is possible by \eqref{e:evrho}. For $\rho<\hat{\rho}$ we define $\rho_k:=(2^k\rho)\wedge\hat{\rho}$ and we adopt the notation 
$k$ in place of $\rho_k$ for the subscriptions.
As above, using \cite[Lemma 4.3]{ContiFocardiIurlano2016} and the triangular inequality we infer
\bm{\nonumber
\|a_k-a_{k+1}\|_{L^\infty({B^+_{\rho_k}},\R^2)}
\leq\\ \frac c{{\rho_k}^2}\int_{\{u=w^+_{k}\}}|w^+_{k}-a_{k}|dx+\frac c{{\rho^2_{k+1}}}\int_{\{u=w^+_{{k+1}}\}}|w^+_{{k+1}}-a_{{k+1}}|dx
\leq c\delta\rho_k^{\frac{p-1}p},
} 
where the last estimate follows by H\"older's inequality, \eqref{e:ppoi}, and \eqref{e:delta}.
Therefore 
\begin{equation}\label{dadkdk1}
|d_k-d_{k+1}|\leq \|a_k-a_{k+1}\|_{L^\infty({{B^+_{\rho_k}}},\R^2)}\leq c\delta\rho_k^{\frac{p-1}p},
\end{equation}
and hence once more by \cite[Lemma 4.3]{ContiFocardiIurlano2016} and by the triangular inequality we conclude
\begin{equation*}
|\beta_k-\beta_{k+1}|\leq c\delta\rho_k^{-\frac{1}p}.
\end{equation*}
Collecting these estimates as $k$ varies we obtain
\[
\rho|\beta_\rho|^p\leq \rho\Big(|\hat\beta|+\sum_{k=0}^{\hat{k}-1}|\beta_k-\beta_{k+1}|\Big)^p\leq c\delta^p+c\rho|\hat{\beta}|^p,
\]
where $\hat{k}$ is the first index 
such that $\rho_{\hat{k}}=\hat{\rho}$ and $\hat{\beta}:=\beta_{\hat{k}}=\beta_{\hat{\rho}}$. This proves \eqref{e:betax0}
as $\rho\to0$ and $\delta\to0$.

We next prove \eqref{e:ax0}. 
{Similarly to the previous estimate, summing (\ref{dadkdk1}) gives 
\begin{equation*}
 |d_\rho -d_{\hat\rho}| \le c \delta \, \hat\rho^{(p-1)/p}
\end{equation*}
for all $0<\rho<\hat\rho\le \rho_\delta$, with $\delta$ arbitrary and $\rho_\delta$ depending only on $\delta$. 
Taking $\rho\to0$ and using (\ref{e:drho}) yields
\begin{equation*}
 \hat\rho^{(1-p)/p} |u^+(x_0) -d_{\hat\rho}| \le c \delta 
\end{equation*}
which, since $\delta$ was arbitrary, proves (\ref{e:ax0}) and therefore (v).}

{
{ At this point we turn to property (ii). }
Korn's inequality implies that
\begin{multline*}
\|\nabla w_\rho^+\|_{L^p(B_{\rho},\R^{2{\times}2})}\leq
  \|\nabla w_\rho^+-\beta_\rho\|_{L^p(B_{\rho},\R^{2{\times}2})}+c\, \rho^{\sfrac2p}|\beta_\rho|\\
  \leq c\,\|e(w_\rho^+)\|_{L^p(B_{\rho},\R^{2{\times}2})}+ c\,\rho^{\sfrac2p}|\beta_\rho|,
\end{multline*}
where $c>0$ is a universal constant. 
{This, together with \eqref{e:evrho} and \eqref{e:betax0}  and
the corresponding estimates for $w_\rho^-$, implies property (ii).}
}

We finally show property (vi). Note that by the trace theorem we have
\bm{
\frac1{\rho}\int_{\Gamma\cap B_\rho}|v^\pm_\rho-u^\pm|d\calH^1\leq \nonumber\\
\frac c{\rho^2}\int_{B_\rho}|v_\rho-u|dx +
\frac c\rho|E(v_\rho-u)|(B_\rho\setminus\Gamma)\leq\\
\frac c{\rho^2}\int_{B_\rho}|v_\rho-u|dx +\frac c\rho\int_{B_\rho}|e(v_\rho)|dx +\frac c{\rho}\int_{B_\rho}|e(u)|dx 
+\frac c\rho\int_{J_u\setminus\Gamma}|[u]|d\calH^1
}
and all terms in the last expression approach $0$ respectively by (iv), \eqref{e:choice1}, \eqref{e:evrho} and \eqref{eq:diffsimm}.

\end{proof}

\section{Integral representation}\label{s:intrep1}
\subsection{Preliminaries}
In this Section we prove Theorem \ref{theorepr}, along the lines of \cite[Section 2.2]{bou-fon-leo-masc}.

Before starting we specify that property \ref{theoreprlsc} means that if $u_j, u\in SBD^p(\Omega)$ obey $u_j\to u$ in $L^1(\Omega,\R^2)$, then
$F(u,A)\le\liminf_{j\to\infty} F(u_j,A)$ for any open set $A$. Property \ref{theoreprlocal}
means that if $u,v\in SBD^p(\Omega)$ obey $u=v$ $\calL^2$-a.e. in $A$, then $F(u,A)=F(v,A)$.
The functions $f$ and $g$ are defined in (\ref{eqdeff}) and (\ref{eqdefg}) below. 

The family of balls contained in $\Omega$ is denoted by
\begin{equation*}
 \calA^*(\Omega):=\left\{ B_\eps(x): x\in\Omega, \eps>0\right\}\,.
\end{equation*}
Let $B\in \calA^*(\Omega)$. We can identify any $u\in SBD^p(B)$
with its zero extension $u\chi_B\in SBD^p(\Omega)$, and correspondingly 
write $F(u,B)$ for $F(u\chi_B,B)$.
By locality, for any other extension the value of the functional is the same.

For $B\in \calA^*(\Omega)$ we define
\begin{equation*}
 m(u,B):=\inf \{ F(w,B): w\in SBD^p(B),\,\, w=u \text{ around } \partial B\}
\end{equation*}
where the condition $w=u$ around $\partial B$ means that a ball $B'\subset\subset B$ exists, so that $w=u$ on $B\setminus B'$.
For $\delta>0$, $A\in \calA(\Omega)$, we set
\begin{align*}
 m^\delta(u,A):=\inf\{ &\sum_{i=1}^\infty m(u, B_i): B_i \in \calA^*, B_i\cap B_j=\emptyset, B_i\subset A, \\
 &\diam(B_i)<\delta, \ \mu(A\setminus \bigcup_{i=1}^\infty B_i)=0\}\,,
\end{align*}
where $\mu:= \calL^2 \LL \Omega + (1+|[u]|) \Huno\LL (J_u\cap \Omega)$. 

Since {$\delta\mapsto m^\delta(u,A)$ is decreasing, we can define} 
\begin{equation*}
  m^*(u,A):= \lim_{\delta\to0} m^\delta(u,A).
 \end{equation*}
{Moreover, we set}
  \begin{equation}\label{eqdeff}
 f(x_0, u_0, \xi) := \limsup_{\eps\to0} \frac{ m(u_0+\xi(\cdot -x_0), B_\eps(x_0))} {\calL^2(B_\eps)}
\end{equation}
\begin{equation}\label{eqdefg}
 g(x_0, a,b,\nu) := \limsup_{\eps\to0} \frac{ m(u_{x_0,a,b,\nu}, B_\eps(x_0))} {2\eps},
\end{equation}
where $u_{x_0,a,b,\nu}$ is defined as 
\begin{equation*}
 u_{x_0,a,b,\nu}(x):=\begin{cases}
              a & \textrm{if }\langle x-x_0,\nu\rangle >0,\\
              b & \textrm{if }\langle x-x_0,\nu\rangle <0.
             \end{cases}
\end{equation*}

 \begin{lemma}\label{lem1}
  For all $u\in SBD^p(\Omega)$ and $A\in\calA(\Omega)$, $F(u,A)=m^*(u,A)$.
 \end{lemma}
\begin{proof}
 By definition, $m(u,B)\le F(u,B)$ for any ball $B$. Since $F(u,\cdot)$ is a measure,
 we obtain $m^\delta(u,A)\le F(u,A)$ for any $\delta>0$. Therefore
 $m^*(u,A)\le F(u,A)$.
 
 To prove the converse inequality, let $\delta>0$, pick countably many balls $B_i^\delta$ as in 
 the definition of $m^\delta(u,A)$, such that
 \begin{equation*}
  \sum_{i=1}^\infty m(u,B_i^\delta) < m^\delta(u,A)+\delta\,.
 \end{equation*}
By the definition of $m$ there are functions $v_i^\delta\in SBD^p(B_i^\delta)$ such that
$v_i^\delta=u$ around $\partial B_i^\delta$ and $F(v_i^\delta, B_i^\delta)\le m(u,B_i^\delta)+\delta \calL^2(B_i^\delta)$.
We define
\begin{equation*}
 v^\delta:=\sum_{i=1}^\infty v_i^\delta \chi_{B_i^\delta} + u \chi_{N_0^\delta}
\end{equation*}
where $N_0^\delta:=\Omega\setminus \cup_i B_i^\delta$.

By the $BD$ compactness theorem $v^\delta\in BD(\Omega)$ and {by the $SBD$ closure theorem (see also \cite[Theorem 11.3]{gbd})} we conclude
that $v^\delta\in SBD^p(\Omega)$ and
\begin{equation*}
 Ev^\delta=\sum_{i=1}^\infty Ev_i^\delta \LL B_i^\delta + Eu \LL N_0^\delta\,,
\end{equation*}
with
\begin{equation*}
 |Ev^\delta|\LL N^\delta=0,\hskip5mm  \mu(N^\delta)=0\,, \hskip1cm F(v^\delta, N^\delta)=0
\end{equation*}
where $N^\delta:=A\cap N_0^\delta$. Further,
\begin{equation*}
 F(v^\delta, A)=\sum_{i=1}^\infty F(v_i^\delta, B_i^\delta)+F(v^\delta, N^\delta) 
 \le m^\delta(u,A)+\delta + \delta \calL^2(A)\,.
\end{equation*}
We claim that $v^\delta\to u$ in $L^1(\Omega,\R^2)$. Since $F(\cdot, A)$ is lower semicontinuous, this will imply
\begin{equation*}
 F(u,A)\le \liminf_{\delta\to0} F(v^\delta,A) 
\le \liminf_{\delta\to0} m^\delta({u},A) =m^*(u,A)\,.
\end{equation*}
To prove  $v^\delta\to u$, we observe that by Poincar\'e's inequality, $\diam B_i^\delta\le \delta$, and $v^\delta=u$ on $\partial B_i^\delta$
we obtain
\begin{equation*}
 \|v^\delta-u\|_{L^1(B_i^\delta,\R^2)} \le c \delta |Ev^\delta-Eu|(B_i^\delta)\,.
\end{equation*}
Therefore
\begin{align*}
 \|v^\delta-u\|_{L^1(\Omega,\R^2)} \le& 
 \sum_i  \|v^\delta-u\|_{L^1(B_i^\delta,\R^2)} \le
 c \delta (|Ev^\delta|(A)+|Eu|(A))\\
 \le &c \delta (F(v^\delta, A)+F(u,A))\,.
\end{align*}
Since $F(v^\delta,A)$ has a finite limit as $\delta\to0$, this proves 
$v^\delta\to u$ in $L^1(\Omega,\R^2)$.
\end{proof}

 \begin{lemma}\label{lem3}
 For any ball $B_r(x_0)\subset\Omega$ and $\delta>0$ sufficiently small we have
 \begin{enumerate}
  \item $\displaystyle \lim_{\delta\to0} m(u, B_{r-\delta}(x_0))=m(u,B_r(x_0))$;
  \item $\displaystyle m(u,B_{r+\delta}(x_0)))\leq m(u,B_r(x_0))+\beta\int_{B_{r+\delta}(x_0)\setminus B_r(x_0)}(1+|e(u)|^p)dx
 +\beta\int_{J_u\cap B_{r+\delta}(x_0)\setminus B_r(x_0)}(1+|[u]|)d\Huno.
$
 \end{enumerate}
 \end{lemma}
\begin{proof}
We drop $x_0$ from the notation.
Choose $v_\delta\in SBD^p(B_{r-\delta})$
 with $v_\delta=u$ around $\partial B_{r-\delta}$ and $F(v_\delta, B_{r-\delta})\le m(u, B_{r-\delta})+\delta$.
We define
\begin{equation*}
 w_\delta(x) := \begin{cases}
       v_\delta(x) & \text{ if } x\in  B_{r-\delta}\,,\\
       u(x) & \text{ if } x\in \Omega\setminus  B_{r-\delta}\,.
      \end{cases}
\end{equation*}
We have
\begin{align*}
 m(u,B_r)\le& F(w_\delta,B_r)\le F(v_\delta, B_{r-\delta}) + 
 F(w_\delta,  B_{r}\setminus  B_{r-\delta})\\
 \le &m(u, B_{r-\delta})+\delta\\
 &+ 
 \beta \int_{B_r\setminus  B_{r-\delta}} (|e(u)|^p+1)dx +\beta\int_{J_u\cap B_r\setminus  B_{r-\delta}}
 (1+|[u]|) d\calH^{n-1}\,.
\end{align*}
Since $(1+|e(u)|^p)\calL^2 + (1+|[u]|)\calH^{1}\LL J_u$ is a bounded measure,  we conclude that
\begin{equation*}
 m(u,B_r)\le \liminf_{\delta\to0} m(u, B_{r-\delta})\,.
\end{equation*}
Conversely, for any $\eps>0$ there 
is $v_\eps\in SBD^p(B_r)$ with $v_\eps=u$ around $\partial B_r$ 
and $F(v_\eps,B_r)\le m(u, B_{r})+\eps$. For $\delta>0$  sufficiently small one has
 $v_\eps=u$ on $B_r\setminus B_{r-2\delta}$ and therefore
 $m(u, B_{r-\delta})\le F(v_\eps, B_{r-\delta}) \le m(u, B_{r})+\eps$. Taking first $\delta\to0$ and then $\eps\to0$ concludes the proof of (i). The proof of (ii) is analogous.
\end{proof}

 \begin{lemma}\label{lem2}
For $\mu$-a.e. $x_0\in\Omega$, 
\begin{equation*}
 \lim_{\eps\to0} \frac{ F(u, B_\eps(x_0))}{\mu(B_\eps(x_0))} = \lim_{\eps\to0} \frac{ m(u, B_\eps(x_0))}{\mu(B_\eps(x_0))}\,.
\end{equation*}
 \end{lemma}
\begin{proof}
 From $m(u,B_\eps(x_0))\le F(u,B_\eps(x_0))$ one
 immediately obtains
 \begin{equation*}
  \limsup_{\eps\to0} \frac{ m(u, B_\eps(x_0))} {\mu(B_\eps(x_0))} \le 
 \limsup_{\eps\to0} \frac{ F(u, B_\eps(x_0))} {\mu(B_\eps(x_0))} 
 \end{equation*}
 for any $x_0\in\Omega$.
To prove the converse inequality, we define for $t>0$ the set
\begin{align*}
 E_t:= \{ x\in \Omega:& \text{ there is } \eps_h\to 0 \text{ such that }  \\
& F(u, B_{\eps_h}(x))>m(u, B_{\eps_h}(x))+ t \mu(B_{\eps_h}(x)) \text{ for all } h\}\,.
\end{align*}
From this definition one immediately has
\begin{equation*}
 \liminf_{\eps\to0} \frac{ F(u, B_\eps(x_0))} {\mu(B_\eps(x_0))} \le
 \liminf_{\eps\to0} \frac{ m(u, B_\eps(x_0))} {\mu(B_\eps(x_0))} + 
 t \hskip4mm\text{ for all } x_0\in\Omega\setminus E_t\,.
\end{equation*}
If we can prove that
\begin{equation}\label{eqmut}
 \mu(E_t)=0\text{ for all } t>0
\end{equation}
then, recalling that $ \lim_{\eps\to0} \frac{ F(u, B_\eps(x_0))} {\mu(B_\eps(x_0))}$ exists $\mu$-almost everywhere,
the proof is concluded.

It remains to prove (\ref{eqmut}) for an arbitrary $t>0$.
For $\delta>0$ we define
\begin{align*}
 X^\delta:=\{ B_\eps(x): &\,\, \eps<\delta, \hskip2mm \overline B_\eps(x)\subset\Omega, \hskip2mm\mu(\partial B_\eps(x))=0,\\
 &F(u, B_\eps(x)) > m(u,B_\eps(x)) + t \mu(B_\eps(x))\}
\end{align*}
and
\begin{align*}
 U^*:= \bigcap_{\delta>0} \{x: \exists \eps>0 \text{ s.t. } B_\eps(x)\in X^\delta\} \,.
\end{align*}
We first show that $E_t\subset U^*$.
Let $x\in E_t$. Then for any $\delta>0$ there is $\eps\in(0,\delta)$ such that
$F(u, B_\eps(x)) > m(u,B_\eps(x)) + t \mu(B_\eps(x))$.
By Lemma \ref{lem3} the function $\eps\to m(u,B_\eps(x))$ is left-continuous; 
$F(u,B_\eps(x))$ is left-continuous because $F(u,\cdot)$ is a measure, therefore the 
same inequality holds for all $\eps'\in (\eps'',\eps)$. In particular, there is one
which additionally  obeys
$\mu(\partial B_\eps(x))=0$.

It remains to show that $\mu(U^*)=0$.
We fix a compact set $K\subset U^*$ and $0<\delta<\eta$. 
Let $U^\eta:=\bigcup\{ B_\eps(x): B_\eps(x)\in X^\eta\}$
and
\begin{align*}
 Y^\delta:= \{ B_\eps(x): \eps<\delta, \overline{B_\eps(x)}\subset U^\eta\setminus K,
 \mu(\partial B_\eps(x))=0\}\,.
\end{align*}
By definition, $X^\delta$ is a fine cover of $K$ and $Y^\delta$ of $U^\eta\setminus K$. Therefore 
there are countably many pairwise disjoint balls $ B_i\in X^\delta$ and $\hat { B}_j\in Y^\delta$
and a set $N$ with $\mu(N)=0$ such that 
\begin{equation*}
 U^\eta=\left( \bigcup_{i\in \N}  B_i\right)
 \cup\left( \bigcup_{j\in \N} \hat { B_j}\right)\cup N\,.
\end{equation*}
Then
\begin{align*}
 F(u, U^\eta)= &\sum_i F(u,  B_i)+\sum_j F(u,\hat B_j) +F(u,N)
\\
\ge & \sum_i (m(u, B_i)+t\mu(B_i)) + \sum_j m (u,\hat B_j)\\
=& \sum_i m(u, B_i) + \sum_j m(u, \hat B_j) + t\mu(\cup_i B_i)\\
 \ge & m^\delta(u, U^\eta) + t \mu(K)
\end{align*}
where in the last step we used the definition of $m^\delta$.
For $\delta\to0$, the definition of $m^*$ and Lemma \ref{lem1} give
\begin{align*}
 F(u, U^\eta)\ge m^*(u, U^\eta) + t\mu(K) = F(u, U^\eta) +t\mu(K)\,.
\end{align*}
Therefore $\mu(K)=0$, and by the regularity of $\mu$ we conclude $\mu(U^*)=0$.

%
%
%
%
%
%
\end{proof}

\subsection{Bounds on the volume term}
In this subsection we identify the volume energy density in the integral 
representation for $F$ to be the function $f$ defined in \eqref{eqdeff}. 
Throughout the whole subsection we consider a fixed map $u\in SBD^p(\Omega)$.
Our first result shows that the local volume energy density can be computed with a
$W^{1,p}$-approximation  to the blow-ups of $u$ (see (\ref{eqconveweps}--\ref{eqconvweps}) below), in the sense that 
\begin{align}\label{equsdfstep12}
   \frac{ dF(u,\cdot)}{d\Ln} (x_0) 
=      \lim_{\eps\to0} \frac{ m\big(w_\eps, B_{\eps}(x_0)\big)}{\calL^2(B_\eps)}.
\end{align}
We shall however not need (\ref{equsdfstep12}), but only the apparently more complex
version in (\ref{equsdfstep1})-(\ref{equsdfstep2}). Taking a diagonal subsequence they imply
(\ref{equsdfstep12}).
\begin{lemma}\label{lemmabdvolpart1}
For $\calL^2$-almost any $x_0\in\Omega$, any $\eps>0$, and any $s\in (0,1)$
there are functions  $w_\eps^s\in  W^{1,p}(B_{s\eps}(x_0);\R^2)$
which obey 
\begin{align}\label{equsdfstep1}
   \frac{ dF(u,\cdot)}{d\Ln} (x_0) 
   \le \liminf_{s\to1}
    \liminf_{\eps\to0} \frac{ m\big(w_\eps^s, B_{s\eps}(x_0)\big)}{\calL^2(B_{s\eps})}
\end{align}
and
\begin{align}\label{equsdfstep2}
 \limsup_{s\to1}
    \limsup_{\eps\to0} \frac{ m\big(w_\eps^s, B_{s^2\eps}(x_0)\big)}{\calL^2(B_{s\eps})}
\le    \frac{ dF(u,\cdot)}{d\Ln} (x_0) 
    \end{align}
and which approximate the affine function $y\mapsto \nabla u(x_0)(y-x_0)+u(x_0)$ in the sense that
\begin{align}\label{eqconveweps}
 \lim_{\eps\to0} \frac{1}{\eps^2} \int_{B_{\eps}(x_0)} |e(w^s_\eps)-e(u)(x_0)|^p dx=0
\end{align}
and
\begin{align}\label{eqconvweps}
 \lim_{\eps\to0} \frac{1}{\eps^{2+p}} \int_{B_{\eps}(x_0)} |w^s_\eps(x)-u(x_0)-\nabla u(x_0)(x-x_0)|^p dx=0\,.
\end{align}
\end{lemma}
We remark that the ball in (\ref{equsdfstep2}) has radius $s^2\eps$ instead of $s\eps$. The estimate would also hold
on $B_{s\eps}$, the variant we chose is more convenient in the proof of Lemma \ref{lem:volLB}.
\begin{proof}
Let $x_0\in \Omega$ be such that
 \begin{equation}\label{eqx0lebeu}
  \lim_{\eps\to0} \frac{1}{\eps^2} \int_{B_\eps(x_0)} |e(u)(x)-e(u)(x_0)|^p dx =0\,,
 \end{equation}
 \begin{equation}\label{eqx0lebju}
  \lim_{\eps\to0} \frac{1}{\eps^2} \int_{B_\eps(x_0)\cap J_u} (1+|[u]|)d\calH^1=0\,,
 \end{equation}
and
 \begin{equation}\label{eqx0leu}
  \lim_{\eps\to0} \frac{1}{\eps^3} \int_{B_\eps(x_0)} |u(x)-u(x_0) - \nabla u(x_0)(x-x_0)|dx=0\,.
 \end{equation}
 By \cite[Th. 7.4]{AmbrosioCosciaDalmaso1997}, $\calL^2$-almost every $x_0$ obeys (\ref{eqx0leu}), the other two are standard.

 By (\ref{eqx0lebju}), for sufficiently small $\eps$ one has $\calH^1(J_u\cap B_\eps(x_0))\le \eta(1-s)\eps/2$,
where $\eta$ is the constant from Theorem~\ref{t:tecnico}.
By Proposition~\ref{p:ricopr} applied to $u-u(x_0)-\nabla u(x_0)(\cdot-x_0)$ there is $\tilde{w}^s_\eps\in 
SBD^p(B_\eps(x_0))\cap W^{1,p}(B_{s\eps}(x_0);\R^2)$ with properties (i)-(vii) and we set $w^s_\eps:=\tilde{w}^s_\eps+u(x_0)+\nabla u(x_0)(\cdot-x_0)$.
In particular, (\ref{eqconveweps}) follows from (\ref{e:volume}) and (\ref{eqx0lebeu}), while (\ref{eqconvweps}) follows from Lemma 
\ref{lemmal1lp} below applied to $\tilde{w}^s_\eps$,
estimating the right-hand side with (\ref{eqconveweps}), (vi), and \eqref{eqx0lebeu}-(\ref{eqx0leu}).

We first prove (\ref{equsdfstep2}).
By the definition of $m$ and the fact that $F(w^s_\eps,\cdot)$ is a measure follows
\begin{align*}
 m(w^s_\eps, B_{s\eps}(x_0))
 \le F(w^s_\eps, B_{s\eps}(x_0))
 \le F(w^s_\eps, B_{\eps}(x_0))\,.
\end{align*}
Let $(B_i)_{i\in\N}$ be the balls from Proposition~\ref{p:ricopr}. For $M\in\N$ we define
\begin{equation*}
 w_\eps^{s,M}:= u + {\chi_{\cup_{i=1}^M\overline B_i}} (w^s_\eps-u)\,.
\end{equation*}
Then $w_\eps^{s,M}\in SBD^p(B_\eps(x_0))$ and $w_\eps^{s,M}\to w^s_\eps$ in $L^1$ as $M\to\infty$.
Further,
\begin{align*}
 F(w_\eps^{s,M},B_{\eps}(x_0))\le &
F(w_\eps^{s,M},B_{\eps}(x_0)\setminus \cup_{i=1}^M\overline B_i)+\sum_{i=1}^M 
 F(w_\eps^{s,M},\overline B_i)\\
\le& F(u,B_{\eps}(x_0)\setminus \cup_{i=1}^M\overline B_i)+\beta\sum_{i=1}^M 
\int_{\overline B_i} (1+|e(w_\eps^s)|^p) dx
\end{align*}
since $w^{s,M}_\eps=w^s_\eps$ is a $W^{1,p}$ function on each $\overline B_i$.
By monotonicity and lower semicontinuity of $F$ we obtain
\begin{align*}
 F(w_\eps^s,B_{\eps}(x_0))\le &
 F(u,B_{\eps}(x_0))+\beta\sum_{i=1}^\infty
\int_{\overline B_i} (1+|e(w_\eps^s)|^p) dx\\
\le &
 F(u,B_{\eps}(x_0))+c \calL^2(\cup_i B_i) (1+|e(u)|^p(x_0)) \\
 &
 + c \int_{B_{\eps}(x_0)} (|e(w_\eps^s)-e(u)(x_0)|^p) dx
\end{align*}
and, recalling  Proposition~\ref{p:ricopr} (i), conclude the proof of (\ref{equsdfstep2}) 
by \eqref{eqconveweps} and \eqref{eqx0lebeu}. 

It remains to prove (\ref{equsdfstep1}).
Let $v_\eps\in SBD^p(B_{s^2\eps}(x_0))$ be such that $v_\eps=w^s_\eps$ around $\partial B_{s^2\eps}(x_0)$
and $F(v_\eps, B_{s^2\eps}) \le m(w^s_\eps, B_{s^2\eps}(x_0))+\eps^3$. 
We define
\begin{equation*}
 \tilde v_\eps(x):=
 \begin{cases}
  v_\eps(x) & \text{ if } x\in B_{s^2\eps}(x_0)\\
  w^s_\eps(x) & \text{ if } x\in B_\eps(x_0)\setminus B_{s^2\eps}(x_0)\,.
 \end{cases}
\end{equation*}
By definition of $m$ and additivity of $F$ we obtain
\begin{align*}
 m(u, B_{\eps}(x_0))\le &F(\tilde v_\eps, B_{\eps}(x_0)) 
 =
 F(\tilde v_\eps, B_{s^2\eps}(x_0))+ 
 F(\tilde v_\eps, B_\eps(x_0)\setminus B_{s^2\eps}(x_0))
\end{align*}
where by locality of $F$ and definition of $v_\eps$
\begin{align*}
 F(\tilde v_\eps, B_{s^2\eps}(x_0))= 
F(v_\eps, B_{s^2\eps}(x_0))\le m(w^s_\eps, B_{s^2\eps}(x_0)) + \eps^3
\end{align*}
and, since $\tilde v_\eps=w^s_\eps$ outside $B_{s^2\eps}(x_0)$
and $\calH^1(J_{\tilde v_\eps}\cap \partial B_{s^2\eps}(x_0))=0$, 
recalling (\ref{e:volume}) we obtain
\begin{align*}
F(\tilde v_\eps, B_\eps(x_0)\setminus B_{s^2\eps}(x_0))\le &
  \beta\int_{B_\eps(x_0)\setminus B_{s^2\eps}(x_0)} (1+|e(w_\eps^s)|^p) dx\\
&+ \beta\int_{J_u\cap B_\eps(x_0)\setminus B_{s^2\eps}(x_0)} (1+|[u]|) d\calH^1\\
\le & c\beta \calL^2(B_\eps) (1-s^4) (1+|e(u)|^p(x_0))\\
&+ c\beta \int_{B_\eps(x_0)} |e(w_\eps^s)(x)-e(u)(x_0)|^p dx\\
&+ \beta\int_{J_u\cap B_\eps(x_0)} (1+|[u]|)d\calH^1\,.
\end{align*}
Dividing by $\calL^2(B_\eps)$ and taking the limit $\eps\to0$ gives
\begin{align*}
 \lim_{\eps\to0} \frac{m(u, B_\eps(x_0))}{\calL^2(B_\eps)}
 \le \liminf_{\eps\to0} \frac{ m(w^s_\eps, B_{s^2\eps}(x_0))}{\calL^2(B_\eps)} 
 + c\beta (1-s^4)(1+|e(u)|^p(x_0)) \,,
\end{align*}
where we used \eqref{eqconveweps} and (\ref{eqx0lebju}).
Recalling Lemma \ref{lem2} we obtain
\begin{align*}
   \frac{ dF(u,\cdot)}{d\Ln} (x_0) = \lim_{\eps\to0} \frac{m(u, B_\eps(x_0))}{\calL^2(B_\eps)}
   \le \liminf_{s\to1}
    \liminf_{\eps\to0} \frac{ m(w^s_\eps, B_{s^2\eps}(x_0))}{\calL^2(B_{s\eps})}\,.
\end{align*}
This concludes the proof of (\ref{equsdfstep1}).
\end{proof}

The next Lemma is a reverse-H\"older estimate for functions with small strain, of the form
$\|v\|_p\le r \|e(v)\|_p + \|v\|_1 r^{-n/p'}$.
\begin{lemma}\label{lemmal1lp}
For any $p\ge 1$ there is $c>0$ (depending on $n$ and $p$) such that for any $v\in W^{1,p}(B_r;\R^n)$ one has
 \begin{equation*}
   \frac{1}{r^{n+p}}  \int_{B_r} |v|^p dx \le c \frac{1}{r^n} \int_{B_r} |e(v)|^p dx + c \left(\frac{1}{r^{n+1}}\int_{B_r} |v|dx\right)^p\,.
 \end{equation*}
\end{lemma}
\begin{proof}
By scaling it suffices to consider $r=1$.
 By Korn's inequality there is an affine function $a$ such that
 \begin{equation*}
     \int_{B_1} |v-a|^p dx \le c \int_{B_1} |e(v)|^p dx \,.
 \end{equation*}
Since $a$ is affine,
 \begin{equation*}
     \int_{B_1} |a|^p dx \le c\left( \int_{B_1} |a| dx\right)^p 
     \le c\left( \int_{B_1} |v| dx\right)^p 
      + c \int_{B_1} |v-a|^p dx\,.
 \end{equation*}
A triangular inequality concludes the proof.
\end{proof}

\begin{lemma}\label{lem:volUB}
 For $\Ln$-a.e. $x_0\in\Omega$,
 \begin{equation*}
  \frac{ dF(u,\cdot)}{d\Ln} (x_0) \le f(x_0, u(x_0), \nabla u(x_0))
 \end{equation*}
where $f$ was defined in (\ref{eqdeff}).
\end{lemma}
\begin{proof}
Let $x_0$, $w_\eps^s$ be as in Lemma \ref{lemmabdvolpart1}, for $s\in (0,1)$.
We choose $v^s_\eps\in SBD^p(B_{s^2 \eps}(x_0))$ such that
$v^s_\eps(x)=u(x_0)+\nabla u(x_0)(x-x_0)$ around $\partial B_{s^2\eps}(x_0)$ and
$F(v^s_\eps,  B_{s^2\eps}(x_0))\le m(u(x_0)+\nabla u(x_0)(\cdot-x_0),  B_{s^2\eps}(x_0))+\eps^3$.
We extend it to $\R^2$ setting it equal to 
$ u(x_0)+\nabla u(x_0)(\cdot-x_0)$ outside $B_{s^2\eps}(x_0)$ and
choose $\varphi\in C^\infty_c(B_{s\eps}(x_0))$ with $\varphi=1$ on $B_{s^2\eps}(x_0)$
and $\|D\varphi\|_\infty \le c/(s(1-s) \eps)$. We define
\begin{align*}
z^s_\eps := \varphi v^s_\eps + (1-\varphi) w_\eps^s\,.
\end{align*}
We remark that $z^s_\eps=v^s_\eps$ on  $B_{s^2\eps}(x_0)$
and $z^s_\eps\in W^{1,p}( B_{s\eps}(x_0)\setminus  B_{s^2\eps}(x_0);\R^2)$.
Then
\begin{align*}
 m(w^s_\eps, B_{s\eps}(x_0)) \le& F(z^s_\eps,  B_{s\eps}(x_0))
 \le F(v^s_\eps,  B_{s^2\eps}(x_0)) + F(z^s_\eps, B_{s\eps}(x_0)\setminus  B_{s^2\eps}(x_0)) \\
  \le &m(u(x_0)+\nabla u(x_0)(\cdot-x_0), B_{s^2\eps}(x_0)) +\eps^3\\
& + \beta \int_{ B_{s\eps}(x_0)\setminus  B_{s^2\eps}(x_0)} (1+|e(z^s_\eps)|^p) dx\,.
\end{align*}
In order to estimate the error term, we observe that in $B_{s\eps}(x_0)\setminus  B_{s^2\eps}(x_0)$ one has
\begin{equation*}
 \nabla z^s_\eps-\nabla u(x_0)= (u(x_0)+\nabla u(x_0)(\cdot-x_0)-w^s_\eps) \nabla\varphi + (1-\varphi) (\nabla u(x_0)-\nabla w^s_\eps) 
\end{equation*}
which implies
\begin{align*}
 \int_{ B_{s\eps}(x_0)\setminus  B_{s^2\eps}(x_0)} (1+|e(z^s_\eps)|^p) dx
 \le& c(1-s)\calL^2(B_{s\eps}) (1+|e(u)|^p(x_0))\\
 &+ c \int_{B_{s\eps}(x_0)} |e(u)(x_0)-e(w_\eps^s)|^p  dx\\
&+c \int_{B_{s\eps}(x_0)}\frac{|u(x_0)+\nabla u(x_0)(\cdot-x_0)-w_\eps^s|^p}{\eps^p s^p(1-s)^p}  dx.
\end{align*}
Therefore
\begin{align*}
 \limsup_{\eps\to0} \frac{F(z^s_\eps, B_{s\eps}(x_0)\setminus  B_{s^2\eps}(x_0)) }{\calL^2(B_{s\eps})}\le c 
 (1-s) (1+|e(u)|^p(x_0))
\end{align*}
and
\begin{align*}
 \limsup_{\eps\to0} \frac{ m(w_\eps^s, B_{s\eps}(x_0)) }{\calL^2(B_{s\eps})}\le &
 \limsup_{\eps\to0} \frac{  m(u(x_0)+\nabla u(x_0)(\cdot-x_0), B_{s^2\eps}(x_0))}{\calL^2(B_{s\eps})}\\
 &+
c (1-s) (1+|e(u)|^p(x_0))\\
=& s^2f(x_0,u_0,\nabla u(x_0)) +
c (1-s) (1+|e(u)|^p(x_0))\,.
\end{align*}
Since $s$ was arbitrary, this concludes the proof.
%
\end{proof}

\begin{lemma}\label{lem:volLB}
 For $\Ln$-a.e. $x_0\in\Omega$,
 \begin{equation*}
  f(x_0, u(x_0), \nabla u(x_0))\le \frac{ dF(u,\cdot)}{d\Ln} (x_0) 
 \end{equation*}
where $f$ was defined in (\ref{eqdeff}).
\end{lemma}
\begin{proof}
We choose $x_0$ and $w_\eps^s$ as in Lemma \ref{lemmabdvolpart1}, for
 $s\in (0,1)$.
We let  $v^s_{\eps}\in SBD^p(B_{s^2\eps}(x_0))$ be such that $v_{\eps}^s=w_\eps^s$ 
around $\partial B_{s^2\eps}(x_0)$
and
$F(v_\eps^s, B_{s^2\eps}(x_0))\le m(w_\eps^s, B_{s^2\eps}(x_0))+\eps^3$, and
extend it to $B_{s\eps}(x_0)$ setting it equal to $w_\eps^s$ outside $B_{s^2\eps}(x_0)$.
We choose $\varphi\in C^\infty_c(B_{s\eps}(x_0))$ with $\varphi=1$ on $B_{s^2\eps}(x_0)$
and $\|D\varphi\|_\infty \le c/(s(1-s)\eps)$ and define
\begin{align*}
z^s_\eps := \varphi  v^s_\eps + (1-\varphi) (u(x_0)+\nabla u(x_0)(x-x_0))\,.
\end{align*}
Then 
\begin{align*}
 m(u(x_0)+\nabla u(x_0)(\cdot -x_0), &B_{s\eps}(x_0))
      \le  F(z^s_\eps, B_{s\eps}(x_0))\\
=& F(v_\eps^s, B_{s^2\eps}(x_0))
+F(z^s_\eps, B_{s\eps}(x_0)\setminus B_{s^2\eps}(x_0))
\\      \le&  m(w_\eps^s, B_{s^2\eps}(x_0))
      +\eps^3 + F(z^s_\eps, B_{s\eps}(x_0)\setminus B_{s^2\eps}(x_0))\,.
     \end{align*}
In order to estimate the error term, we observe that in $B_{s\eps}(x_0)\setminus  B_{s^2\eps}(x_0)$ one has
\begin{equation*}
 \nabla z^s_\eps-\nabla u(x_0)= -(u(x_0)+\nabla u(x_0)(\cdot-x_0)-w_\eps^s) \nabla\varphi + \varphi (\nabla w_\eps^s-\nabla u(x_0)) 
\end{equation*}
which leads as in the proof of Lemma \ref{lem:volUB} to
\begin{align*}
 \limsup_{\eps\to0} \frac{F(z^s_\eps, B_{s\eps}(x_0)\setminus  B_{s^2\eps}(x_0)) }{\calL^2(B_{s\eps})}\le c 
 (1-s) (1+|e(u)|^p(x_0))\,.
\end{align*}
     We conclude that for any $s\in(0,1)$
 \begin{align*}
&\limsup_{\eps\to0}      \frac{ m(u(x_0)+\nabla u (x_0)(\cdot -x_0), B_{s\eps}(x_0))} {\Ln(B_{s\eps})}\\
          \le& \limsup_{\eps\to0}  \frac{ m(w_\eps^s, B_{s^2\eps}(x_0))} {\Ln(B_{s\eps})}
+c  (1-s) (1+|e(u)|^p(x_0))\,.
          \end{align*}
          Since $s$ was arbitrary, this concludes the proof.
\end{proof}

\subsection{Bounds on the surface term}\label{s:surface}

In the current subsection we identify the function $g$ in \eqref{eqdefg} to be the surface 
energy density in the integral representation of $F$.
As above, we work with a fixed map $u\in SBD^p(\Omega)$.

We first prove a technical result. 
\begin{lemma}\label{lem:surftech}
 For $\Huno$-a.e. $x_0\in J_u$  
 the functions $v_{2\eps}\in SBV^p(B_{2\eps}(x_0),\R^2)$ 
 introduced in Lemma~\ref{lem:refl} 
 satisfy for all $t\in(0,2)$
 \begin{equation}\label{e:mweps}
  \frac{ dF(u,\cdot)}{d\Huno\LL J_u} (x_0) =\lim_{\eps\to 0}\frac{m(v_{2\eps},B_{t\eps}(x_0))}{2t\eps}.
 \end{equation}
\end{lemma}
\begin{proof} 
It suffices to consider points $x_0$ such that
the conclusions of Lemmata~\ref{lem:refl} and \ref{lem2} hold true,
 the Radon-Nikodym derivative $\frac{dF(u,\cdot)}{d\Huno\res J_u}(x_0)$ exists finite,
 \begin{equation}\label{e:cond1}
 \lim_{\eps\to 0}\frac{\mu(B_\eps(x_0))}{2\eps}=1+|[u](x_0)|,\quad
  \end{equation}
  and
 \begin{equation}\label{e:cond2}
 \lim_{\eps\to 0}\Big(\frac{1}{\eps}\int_{B_\eps(x_0)}|e(u)|^pdx+
 \frac{1}{\eps^2}\int_{B_\eps(x_0)}|u(x)-\uabn|dx\Big)=0,
 \end{equation}
where $u_{x_0}$ is the piecewise constant function defined in \eqref{e:uxzero}.
In view of all these choices and thanks to Lemma~\ref{lem2} we may conclude that
 \begin{equation}\label{e:FHuno}
  \frac{dF(u,\cdot)}{d\Huno\res J_u}(x_0)=\lim_{\eps\to 0}\frac{F(u,B_\eps(x_0))}{2\eps} 
  =\lim_{\eps\to 0}\frac{m(u,B_\eps(x_0))}{2\eps}.
 \end{equation}
 
For $\eps>0$ small enough the function 
$v_{2\eps}$ introduced in Lemma~\ref{lem:refl} belongs to 
$SBD^p(B_{4\eps}(x_0))\cap SBV^p(B_{2\eps}(x_0),\R^2)$ 
and it satisfies properties (i)-(vi). We set $w_\eps:=v_{2\eps}$,
we are left with proving that for all $t\in(0,2)$
\begin{equation}\label{e:wepsUB}
\frac{ dF(u,\cdot)}{d\Huno\LL J_u} (x_0) \ge 
\limsup_{\eps\to 0}\frac{m(w_\eps,B_{t\eps}(x_0))}{2t\eps},
\end{equation}
\begin{equation}\label{e:wepsLB}
\frac{ dF(u,\cdot)}{d\Huno\LL J_u} (x_0) \le 
\liminf_{\eps\to 0}\frac{m(w_\eps,B_{t\eps}(x_0))}{2t\eps}.
\end{equation}
For the sake of notational simplicity we shall prove inequalities \eqref{e:wepsUB} 
and \eqref{e:wepsLB} only for $t=1$. 

We start off with \eqref{e:wepsUB}. Let $(\eps_j)_j$ be a sequence such that
\begin{equation}\label{e:epsj}
\lim_{j\to\infty}\frac{m(w_{\eps_j},B_{\eps_j}(x_0))}{2\eps_j}=
\limsup_{\eps\to 0}\frac{m(w_\eps,B_{\eps}(x_0))}{2\eps}.
\end{equation}
Items (iii) and (iv) in Lemma~\ref{lem:refl} and the Coarea formula yield for a subsequence 
not relabeled for convenience that for $\calL^1$-a.e. $s\in(0,1)$ 
\begin{equation}\label{e:coareasurf2}
\lim_j\frac{1}{\eps_j}\int_{\partial B_{s\eps_j}(x_0)\cap\{u\neq w_{\eps_j}\}}\big(1+|u- w_{\eps_j}|\big)d\Huno=0,
\end{equation}
\begin{equation}\label{e:coareasurf1}
\mu\big(\partial B_{s\eps_j}(x_0)\big)=\calH^1\big(\partial B_{s\eps_j}(x_0)\cap J_{ w_{\eps_j}}\big)=0.
\end{equation}
We choose $z_j\in SBD^p(B_{s\eps_j}(x_0))$ such that
 $z_j=u$ around $\partial B_{s\eps_j}(x_0)$ and 
 \[
  F(z_j,B_{s\eps_j}(x_0))\leq m(u,B_{s\eps_j}(x_0))+\eps_j^2,
 \]
and define
\begin{equation*}
\zeta_j:=\begin{cases}
 z_j & B_{s\eps_j}(x_0) \cr
  w_{\eps_j} & B_{\eps_j}(x_0)\setminus \overline{B_{s\eps_j}(x_0)}.
     \end{cases}
 \end{equation*}
The definition of $z_j$, the growth conditions in \eqref{e:growthF}, and the locality of $F$ yield
\begin{multline*}
 m( w_{\eps_j},B_{\eps_j}(x_0))\leq F(\zeta_j,B_{\eps_j}(x_0))\\
 \leq F(z_j,B_{s\eps_j}(x_0))
 +\underbrace{\beta\int_{B_{\eps_j}(x_0)\setminus B_{s\eps_j}(x_0)}(1+|e( w_{\eps_j})|^p)\,dx}_{=:I_j^{(1)}}\\
 +\underbrace{\beta\int_{\partial B_{s\eps_j}(x_0)\cap\{u\neq w_{\eps_j}\}}(1+|u- w_{\eps_j}|)d\Huno}_{=:I_j^{(2)}}
 +\underbrace{\beta\int_{(B_{\eps_j}(x_0)\setminus \overline{B_{s\eps_j}(x_0)})\cap J_{ w_{\eps_j}}}(1+|[ w_{\eps_j}]|)d\Huno}_{=:I_j^{(3)}}\\
 \leq m(u,B_{s\eps_j}(x_0))+\eps_j^2+I_j^{(1)}+I_j^{(2)}+I_j^{(3)}.
 \end{multline*}
We note that $I_j^{(1)}$ and $I_j^{(2)}$ are $o(\eps_j)$ as $j\to \infty$ thanks to Lemma~\ref{lem:refl} (ii) 
and \eqref{e:coareasurf2}, respectively. Instead, employing Lemma~\ref{lem:refl} (vi) and \eqref{e:cond1} to bound $I_j^{(3)}$ we infer that 
\begin{multline}\label{e:Itrej}
 \limsup_{j\to\infty}\frac{I_j^{(3)}}{2\eps_j}\leq \limsup_{j\to\infty}
 \frac\beta{2\eps_j}\int_{(B_{\eps_j}(x_0)\setminus {B_{s\eps_j}(x_0)}\cap J_{u}}(1+|[u]|)d\Huno\\
 =\beta \limsup_{j\to\infty} 
 \frac{\mu\big(B_{\eps_j}(x_0)\setminus {B_{s\eps_j}(x_0)}\cap J_{u}\big)}{2\eps_j}
 =(1-s)\beta(1+|[u](x_0)|).
\end{multline}
Therefore, by \eqref{e:FHuno} we conclude
\begin{multline*}
\lim_{j\to\infty}\frac{m(w_{\eps_j},B_{\eps_j}(x_0))}{2\eps_j}\leq 
 \liminf_{j\to\infty}\frac{m(u,B_{s\eps_j}(x_0))}{2\eps_j}+(1-s)\beta(1+|[u](x_0)|)\\
 =s\frac{dF(u,\cdot)}{d\Huno\LL J_u} (x_0)+(1-s)\beta(1+|[u](x_0)|).
\end{multline*}
Estimate \eqref{e:wepsUB} follows at once by \eqref{e:epsj} and by letting $s\uparrow 1$ in the 
last inequality.

Let now $(\eps_j)_j$ be a sequence such that
\begin{equation}\label{e:epsj2}
\lim_{j\to\infty}\frac{m(w_{\eps_j},B_{\eps_j}(x_0))}{2\eps_j}=
\liminf_{\eps\to 0}\frac{m(w_\eps,B_{\eps}(x_0))}{2\eps}.
\end{equation}
Let $\lambda\in(1,2)$, arguing as for \eqref{e:coareasurf2} and \eqref{e:coareasurf1},  
up to a subsequence depending on $\lambda$ and not relabeled for convenience we may assume that for 
$\calL^1$-a.e. $s\in(0,1)$
\begin{equation}\label{e:coareasurf2b}
\lim_{j\to\infty}\frac{1}{\eps_j}\int_{\partial B_{s\lambda\eps_j}(x_0)\cap\{u\neq w_{\eps_j}\}}
\big(1+|u-w_{\eps_j}|\big)d\Huno=0,
 \end{equation}
 and
 \begin{equation}\label{e:coareasurf1b}
\mu\big(\partial B_{s\lambda\eps_j}(x_0)\big)=
\calH^1\big(\partial B_{s\lambda\eps_j}(x_0)\cap J_{w_{\eps_j}}\big)=0.
 \end{equation}
Given $z_j\in SBD^p(B_{s\lambda\eps_j}(x_0))$ with $z_j=w_{\eps_j}$ around $\partial B_{s\lambda\eps_j}(x_0)$
and such that 
\[
 F(z_j,B_{s\lambda\eps_j}(x_0))\leq m(w_{\eps_j},B_{s\lambda\eps_j}(x_0))+\eps_j^2,
\]
define
\begin{equation*}
\zeta_j:=\begin{cases}
 z_j & B_{s\lambda\eps_j}(x_0) \cr
 u & B_{\lambda\eps_j}(x_0)\setminus \overline{B_{s\lambda\eps_j}(x_0)}.
     \end{cases}
 \end{equation*}
Using $\zeta_j$ as a test field for $m(u,B_{\lambda\eps_j}(x_0))$, by the locality of 
$F$ and its growth conditions in \eqref{e:growthF} 
\begin{multline*}
 m(u,B_{\lambda\eps_j}(x_0))\leq F(\zeta_j,B_{\lambda\eps_j}(x_0))\leq 
 m(w_{\eps_j},B_{s\lambda\eps_j}(x_0))+\eps_j^2\\
 +\underbrace{\beta\int_{B_{\lambda\eps_j}(x_0)}(1+|e(u)|^p)\,dx}_{I_j^{(4)}}
 +\underbrace{\beta\int_{\partial B_{s\lambda\eps_j}(x_0)\cap\{u\neq w_{\eps_j}\}}
 (1+|u-w_{\eps_j}|)d\Huno}_{I_j^{(5)}}\\
 +\underbrace{\beta\int_{(B_{\lambda\eps_j}(x_0)\setminus \overline{B_{s\lambda\eps_j}(x_0)})\cap 
 J_{u}}(1+|[u]|)d\Huno}_{I_j^{(6)}}.
 \end{multline*}
The terms $I_j^{(4)}$ and $I_j^{(5)}$ are $o(\eps_j)$ by \eqref{e:cond2} and
\eqref{e:coareasurf2b}, respectively. 
The term  $I_j^{(6)}$ can be estimated thanks to \eqref{e:cond1}. 
Hence, we get by \eqref{e:FHuno}
\begin{equation}\label{e:basta}
 \frac{dF(u,\cdot)}{d\Huno\res J_u}(x_0)=\limsup_{j\to\infty}\frac{m(u,B_{\lambda\eps_j}(x_0))}{2\lambda\eps_j}
 \leq\limsup_{j\to\infty}\frac{m(w_{\eps_j},B_{s\lambda\eps_j}(x_0))}{2\lambda\eps_j}.
\end{equation}
Next, by choosing $s\in(0,1)$ for which \eqref{e:coareasurf2b} and \eqref{e:coareasurf1b} hold and  
$s\lambda>1$, we may use Lemma~\ref{lem3}(ii) to infer 
\begin{align}
\nonumber
m(w_{\eps_j},B_{s\lambda\eps_j}(x_0))\leq& m(w_{\eps_j},B_{\eps_j}(x_0))
 +
 \beta\int_{B_{s\lambda\eps_j}(x_0)\setminus B_{\eps_j}(x_0)}(1+|e(w_{\eps_j})|^p)dx
 \\  &
 +
 \beta\int_{(B_{s\lambda\eps_j}(x_0)\setminus B_{\eps_j}(x_0))\cap J_{w_{\eps_j}}}(1+|[w_{\eps_j}]|)
 d\Huno
 .\label{e:basta2}
\end{align}
Clearly, the first integral is $o(\eps_j)$ by Lemma~\ref{lem:refl} (ii), while the other one can be dealt with as 
$I_j^{(3)}$ in \eqref{e:Itrej}. Thus, \eqref{e:basta} and \eqref{e:basta2} give
\[
 \frac{dF(u,\cdot)}{d\Huno\res J_u}(x_0)\leq\frac 1\lambda\lim_{j\to\infty}
 \frac{m(w_{\eps_j},B_{\eps_j}(x_0))}{2\eps_j}+(s\lambda-1)\beta(1+|[u](x_0)|).
\]
In conclusion, by taking into account \eqref{e:epsj2}, we deduce \eqref{e:wepsLB} by taking first the 
limit as $s\uparrow 1$, for $s\in(0,1)$ chosen as explained above, and then as $\lambda\downarrow 1$ in the 
latter inequality.
\end{proof}

We are now ready to show that the function $g$ in \eqref{eqdefg} is the surface energy density 
of $F$. This task shall be accomplished by proving two inequalities.

\begin{lemma}\label{lem:surfUB}
 For $\Huno$-a.e. $x_0\in J_u$,
 \begin{equation*}
  \frac{ dF(u,\cdot)}{d\Huno\LL J_u} (x_0) \le g(x_0, u^+(x_0), u^-(x_0),\nu_u(x_0))
 \end{equation*}
where $g$ was defined in \eqref{eqdefg}.
\end{lemma}
\begin{proof}
We consider the same $x_0$ as in Lemma \ref{lem:surftech}.
In view of (\ref{e:mweps}) and the definition of $g$ in \eqref{eqdefg} it suffices to show that
\begin{equation}\label{e:FHunob}
 \lim_{\eps\to 0}\frac{m(w_\eps,B_{\eps}(x_0))}{2\eps}
 \leq \limsup_{\eps\to 0}\frac{m(u_{x_0},B_{\eps}(x_0))}{2\eps},
 \end{equation}
where $w_\eps$ is the function introduced in Lemma \ref{lem:surftech}.
To prove such a claim consider any sequence $(\eps_j)_j$, we have that for $\calL^1$-a.e. $s\in(0,1)$ 
 \begin{equation}\label{e:coareasurf1c}
  \mu\big(\partial B_{s\eps_j}(x_0)\big)=\calH^1\big(\partial B_{s\eps_j}(x_0)\cap J_{w_j}\big)=0,
 \end{equation}
where we have set $w_j:=w_{\eps_j}$. 
  
Fix $s\in(0,1)$ as above and a test field $z_j\in SBD^p(B_{s\eps_j}(x_0))$ with 
$z_j=\uabn$ on $\partial B_{s\eps_j}(x_0)$ such that
 \[
  F(z_j,B_{s\eps_j}(x_0))\leq m(\uabn,B_{s\eps_j}(x_0))+\eps_j^2.
 \]
Consider a cut-off function $\varphi\in C^\infty_c(B_{\eps_j}(x_0),[0,1])$ such that $\varphi\equiv 1$
on $B_{s\eps_j}(x_0)$ and $\|\nabla \varphi\|_{L^\infty}\leq \frac2{(1-s)\eps_j}$. Define
$\zeta_j:=\varphi\,z_j+(1-\varphi)w_j$, with the convention that $z_j$ is 
extended equal to $\uabn$ outside $B_{s\eps_j}(x_0)$.
Therefore, by using $\zeta_j$ as a test field for $m(w_j,B_{\eps_j}(x_0))$ we infer from 
the growth condition in \eqref{e:growthF} and the locality of $F$ 
\begin{multline}\label{e:stimasurf1}
 m(w_j,B_{\eps_j}(x_0))\leq F(\zeta_j,B_{\eps_j}(x_0))\leq F(z_j,B_{s\eps_j}(x_0))\\
 +\underbrace{C\int_{B_{\eps_j}(x_0)\setminus B_{s\eps_j}(x_0)}
 (1+|e(w_j)|^p)\,dx}_{=:I_j^{(7)}}
 +\underbrace{\frac{C}{((1-s)\eps_j)^p}\int_{B_{\eps_j}(x_0)\setminus B_{s\eps_j}(x_0)}|w_j-\uabn|^pdx}_{=:I_j^{(8)}}\\
 +\underbrace{C\,\Huno\big((B_{\eps_j}(x_0)\setminus B_{s\eps_j}(x_0))\cap J_{\zeta_j}\big)}_{=:I_j^{(9)}}
 +\underbrace{C\int_{(B_{\eps_j}(x_0)\setminus B_{s\eps_j}(x_0))\cap J_{\zeta_j}}|[\zeta_j]|d\Huno}_{=:I_j^{(10)}}\\
 \leq m(\uabn,B_{s\eps_j}(x_0))+\eps_j^2+I_j^{(7)}+I_j^{(8)}+I_j^{(9)}+I_j^{(10)},
 \end{multline}
 with $C=C(\beta,p)>0$.
 
By taking into account Lemma~\ref{lem:refl} (ii) and (v) we deduce that 
$I_j^{(7)}+I_j^{(8)}=o(\eps_j)$ as $j\to \infty$. 
Moreover, as 
\[
\Huno((B_{\eps_j}(x_0)\setminus B_{s\eps_j}(x_0))\cap J_{\zeta_j}\setminus (J_{\uabn}\cup J_{w_j}))=0,
\]
item (i) in Lemma~\ref{lem:refl} together with \eqref{e:cond1} 
give
 \[
  \limsup_{j\to\infty}\frac{I_j^{(9)}}{2\eps_j}\leq C(1-s)(1+|[u](x_0)|).
 \]
Furthermore, for $\Huno$-a.e. $x\in J_{\zeta_j}\cap (B_{\eps_j}(x_0)\setminus \overline{B_{s\eps_j}(x_0)})$ 
it holds
 \begin{equation*}
  |[\zeta_j]|\leq |[\uabn]|\chi_{J_{\uabn}\cap J_{\zeta_j}}+|[w_j]|\chi_{J_{w_j}\cap J_{\zeta_j}}\leq
  2|[\uabn]|\chi_{J_{\zeta_j}}+|[w_j]-[\uabn]|\chi_{J_{w_j}}.
 \end{equation*}
 In turn the latter inequality implies by \eqref{e:cond1} 
 and \eqref{e:coareasurf1c} 
 \begin{multline*}
  \limsup_{j\to\infty}\frac{I_j^{(10)}}{2\eps_j}\leq C(1-s)|[u](x_0)|\\
  +C\,\limsup_{j\to\infty}\frac{1}{2\eps_j}
  \int_{(B_{\eps_j}(x_0)\setminus B_{s\eps_j}(x_0))\cap J_{\zeta_j}}
  (|{[w_j]-[u](x_0)}|)\,d\Huno\\
  \leq C(1-s)|[u](x_0)|,
 \end{multline*}
thanks to item (vi) in Lemma~\ref{lem:refl}.

Finally, we obtain from \eqref{e:stimasurf1}
\begin{multline*}
\liminf_{j\to\infty}\frac{m(w_j,B_{\eps_j}(x_0))}{2\eps_j}\leq
 s\limsup_{j\to\infty}\frac{m(\uabn,B_{s\eps_j}(x_0))}{2s\eps_j}+C(1-s)(1+|[u](x_0)|)\\
 \leq s\limsup_{\eps\to0}\frac{m(\uabn,B_{\eps}(x_0))}{2\eps}+C(1-s)(1+|[u](x_0)|),
\end{multline*}
 and the claim in \eqref{e:FHunob} follows at once by letting $s\to 1$ in the inequality above.
 \end{proof}
The reverse inequality is established arguing in an analogous fashion, therefore we provide 
a more concise proof. 
\begin{lemma}\label{lem:surfLB}
 For $\Huno$-a.e. $x_0\in J_u$,
 \begin{equation*}
  \frac{ dF(u,\cdot)}{d\Huno\LL J_u} (x_0) \ge g(x_0, u^+(x_0), u^-(x_0),\nu_u(x_0))
 \end{equation*}
where $g$ was defined in \eqref{eqdefg}.
\end{lemma}
\begin{proof}
We consider the same points $x_0$ as in Lemma \ref{lem:surftech}.
Take any infinitesimal sequence $(\eps_j)_j$ such that
  \[
   g(x_0, u^+(x_0), u^-(x_0),\nu_u(x_0))=\lim_{j\to\infty}\frac{m(\uabn,B_{\eps_j}(x_0))}{2\eps_j},
  \]
 and recall that \eqref{e:coareasurf1c} is valid for $\calL^1$-a.e. $s\in(0,1)$ (as usual $w_j=w_{\eps_j}$). 
 Having fixed such an $s$, let $z_j\in SBD^p(B_{s\,\eps_j}(x_0))$ with
 $z_j=w_j$ on $\partial B_{s\eps_j}(x_0)$ be such that
\[
 F(z_j,B_{s\eps_j}(x_0))\leq m(w_j,B_{s\eps_j}(x_0))+\eps_j^2.
\]
Let $\varphi\in C^\infty_c(B_{\eps_j}(x_0),[0,1])$ be a cut-off function such that $\varphi\equiv 1$
on $B_{s\eps_j}(x_0)$ and $\|\nabla \varphi\|_{L^\infty}\leq \frac2{(1-s)\eps_j}$. Define
$ \zeta_j:=\varphi\,z_j+(1-\varphi)\uabn$,
with the convention that $z_j$ is extended equal to $w_j$ outside $B_{s\eps_j}(x_0)$.
By using $\zeta_j$ as a test field for $m(\uabn,B_{\eps_j}(x_0))$ we infer from the growth 
condition in \eqref{e:growthF} and the locality of $F$
\begin{multline*}
 m(\uabn,B_{\eps_j}(x_0))\leq F(\zeta_j,B_{\eps_j}(x_0))\leq m(w_j,B_{s\eps_j}(x_0))+\eps_j^2\\
 + C\int_{B_{\eps_j}(x_0)\setminus B_{s\,\eps_j}(x_0)}(1+|e(w_j)|^p)\,dx
 + \frac{C}{((1-s)\eps_j)^p}\int_{B_{\eps_j}(x_0)\setminus B_{s\eps_j}(x_0)}|w_j-\uabn|^pdx\\
+ C\int_{(B_{\eps_j}(x_0)\setminus B_{s\eps_j}(x_0))\cap J_{\zeta_j}}(1+|[\zeta_j]|)d\Huno,
 \end{multline*}
 where $C=C(\beta,p)>0$.
Arguing as in the corresponding estimate in Lemma~\ref{lem:surfUB} (cf. \eqref{e:stimasurf1}), and 
by taking into account the choice of $(\eps_j)_j$ we conclude that
\begin{multline*}
 g(x_0, u^+(x_0), u^-(x_0),\nu_u(x_0))\leq\liminf_{j\to\infty}\frac{m(w_j,B_{s\eps_j}(x_0))}{2\eps_j}
 +C(1-s)(1+|[u](x_0)|)\\
 =s\frac{ dF(u,\cdot)}{d\Huno\LL J_u} (x_0)+C(1-s)(1+|[u](x_0)|).
\end{multline*}
The last equality follows from \eqref{e:mweps}.
The conclusion is achieved by letting $s\uparrow 1$ in the last inequality, with $s\in(0,1)$ 
satisfying \eqref{e:coareasurf1c}.
\end{proof}

\subsection{Proof of Theorem \ref{theorepr}}
\begin{proof}[Proof of Theorem \ref{theorepr}]
The conclusion straightforwardly follows by Lemmata \ref{lem:volUB}, \ref{lem:volLB}, \ref{lem:surfUB}, and \ref{lem:surfLB}. 
\end{proof}
\begin{proposition}\label{prop:growth}
The assertion in Theorem~\ref{theorepr} holds also if property (iv) is replaced by the weaker
\begin{enumerate}
 \item[(iv')] There are $\alpha,\beta>0$ such that for any $u\in SBD^p(\Omega)$, any $B\in \calB(\Omega)$,
  \begin{align*}
   &\alpha \Bigl(\int_B |e(u)|^pdx + \Huno(J_u\cap B)\Bigr) \le
    F(u,B)\notag\\
    \le &
  \beta \Bigl(\int_B (|e(u)|^p+1)dx + \int_{J_u\cap B} (1+|[u]|) d\Huno\Bigr).
   \end{align*}
   \end{enumerate}
   \end{proposition}
   \begin{proof}
Given $F$ satisfying properties (i)-(iii) and (iv'), we define for $\delta>0$ a functional $F_\delta:SBD^p(\Omega)\times \calB(\Omega)\to[0,\infty)$ by
$$F_\delta(u,B):=F(u,B)+\delta\int_{J_u\cap B}|[u]|d\Huno,$$
for $u\in SBD^p(\Omega)$ and $B\in\calB(\Omega)$. Since $F_\delta$ satisfies properties (i)-(iv), there are two functions $f$ and $g_\delta$
such that $F_\delta$ can be represented as in \eqref{e:fg}. 
The family of functionals $F_\delta$ is pointwise increasing in $\delta$, therefore there exists the pointwise limit $g$
of $g_\delta$ as $\delta\to 0$. We conclude that the representation \eqref{e:fg} holds for $F$ with densities $f$ and $g$.
\end{proof}
\begin{remark}
 Since $F$ is lower semicontinuous on $W^{1,p}$, the integrand $f$ is quasiconvex \cite{AcerbiFusco84,Marcellini1985}. Since $F$ is lower semicontinuous on
 piecewise constant functions, $g$ is $BV$-elliptic \cite{AmbrosioBraides1990a,AmbrosioBraides1990b}.
\end{remark}

\begin{remark}
If the functional $F$ additionally obeys
\[
F(u+I,B)=F(I,B),
\]
for every $u\in SBD^p(\Omega)$, every ball $B\subset\Omega$, and every affine function $I$ such that $e(I)=0$, then 
there are two functions $f:\Omega\times \R^{2\times 2}\to [0,\infty)$ and
$g:\Omega\times\R^2\times S^1\to [0,\infty)$ such that
\[
 F(u,B)=\int_B f(x, e(u(x))) dx + \int_{B\cap J_u} g(x,[u](x),\nu_u(x)) d\Huno\,.
\] 
\end{remark}

\begin{remark}\label{r:Fgradiente}
A growth condition on the volume part of the type of \eqref{e:growthF} alone does 
not force the energy density to depend only on $e(u)$. 
As an example, the integrand $f:\R^{2\times 2}\to [0,\infty)$ defined by
\[
f(\xi):={(\xi_{11}+\xi_{22})^2}+\sqrt{(\xi_{12}^2+\xi_{21}^2)^2+1}-2\det(\xi)
\]
satisfies 
\[
{\frac18}|\xi+\xi^T|^2\leq f(\xi)\leq \frac14|\xi+\xi^T|^2+1
\]
for every $\xi\in\R^{2\times 2}$, but evidently $f(\xi)$ depends also on the
skew-symmetric part $\xi-\xi^T$. At the same time, $f$ is quasiconvex.
We do not know if there is $g$ such that 
the functional $F$ defined as in (\ref{e:fg}) 
satisfies the growth condition \eqref{e:growthF} and is lower semicontinuous.
\end{remark}

\section*{Acknowledgments} 
F.~Iurlano wishes to thank Gianni Dal Maso for an interesting discussion.
This work was partially supported 
by the Deutsche Forschungsgemeinschaft through the Sonderforschungsbereich 1060 
{\sl ``The mathematics of emergent effects''}, project A6.
S.~Conti thanks the University of Florence for
the warm hospitality of the DiMaI ``Ulisse Dini'', where part of this work was 
carried out. 
M.~Focardi and F.~Iurlano are members of the Gruppo Nazionale per
l'Analisi Matematica, la Probabilit\`a e le loro Applicazioni (GNAMPA)
of the Istituto Nazionale di Alta Matematica (INdAM).


\end{document}